\documentclass[12pt]{amsart}

\usepackage{amsmath,amssymb,amsthm,amsfonts,enumerate,hhline,bm,fullpage,soul}
\usepackage{graphicx}
\usepackage[colorlinks=true,citecolor=ForestGreen,linkcolor=purple]{hyperref}
\usepackage{verbatim}
\usepackage{color}
\usepackage[top=1in, bottom=1in, left=0.9in, right=0.9in]{geometry}
\usepackage{mathrsfs}
\usepackage[dvipsnames]{xcolor}
\usepackage{bigstrut}
\usepackage{float}

\restylefloat{table}

\makeatletter
\@namedef{subjclassname@2020}{\textup{2020} Mathematics Subject Classification}
\makeatother

\usepackage{hyperref}
\usepackage[all]{xy}
\usepackage{xcolor} 
\usepackage{framed} 
\usepackage{mathdots}
\usepackage{multirow,caption}
\usepackage{footnote}
\usepackage{adjustbox}

\newcommand{\veq}{\mathrel{\rotatebox{90}{$=$}}}

\theoremstyle{theorem}
\newtheorem{theorem}{Theorem}[section]

\newtheorem{corollary}[theorem]{Corollary}
\newtheorem{lemma}[theorem]{Lemma}

\newtheorem{proposition}[theorem]{Proposition}
\theoremstyle{definition}

\newtheorem{example}[theorem]{Example}
\newtheorem{remark}[theorem]{Remark}

\newtheorem{conjecture}[theorem]{Conjecture}
\numberwithin{equation}{section}

\theoremstyle{theorem}
\newtheorem{question}{Question}

\DeclareMathOperator{\Hom}{Hom}%
\DeclareMathOperator{\coker}{Coker}%
\DeclareMathOperator{\codim}{codim}%
\DeclareMathOperator{\Char}{char}%
\DeclareMathOperator{\Ext}{Ext}%
\DeclareMathOperator{\Tor}{Tor}%
\DeclareMathOperator{\het}{ht}%
\DeclareMathOperator{\projdim}{pdim}%
\DeclareMathOperator{\rank}{rk}%
\DeclareMathOperator{\reg}{reg}%
\newcommand{\fm}{\mathfrak{m}}%
\usepackage{multicol}

\newcommand{\llar}{-\kern-5pt-\kern-5pt\longrightarrow}

\def\restr{{\kern-1pt\restriction\kern-1pt}}

\usepackage[normalem]{ulem}


\begin{document}

\title{Lower bounds on Betti numbers}

\author[Boocher]{Adam Boocher}
\address{Department of Mathematics, University of San Diego, San Diego, CA 92110, USA}
\email{aboocher@sandiego.edu}

\author[Grifo]{Elo\'isa Grifo}
\address{Department of Mathematics, University of Nebraska -- Lincoln, Lincoln, NE 68588, USA}
\email{grifo@unl.edu}

\dedicatory{Dedicated to David~Eisenbud on the occasion of his 75th birthday.}

\maketitle

\begin{abstract}
	We survey recent results on bounds for Betti numbers of modules over polynomial rings, with an emphasis on lower bounds. Along the way, we give a gentle introduction to free resolutions and Betti numbers, and discuss some of the reasons why one would study these.
\end{abstract}

\section{Introduction}

Consider a polynomial ring over a field $k$, say $R=k[x_1, \ldots, x_n]$. When studying finitely generated graded modules $M$ over $R$, there are many important invariants we may consider, with the Betti numbers of $M$, denoted $\beta_i(M)$, being among some of the richest. The Betti numbers are defined in terms of generators and relations (see Section \ref{Section:what is a free resolution}), with $\beta_0(M)$ being the number of minimal generators of $M$, $\beta_1(M)$ the number of minimal relations on these generators, and so on. Despite this simple definition, they encode a great deal of information. For instance, if one knows the Betti numbers\footnote{Really, we mean the graded Betti numbers of $M$, to be defined in Section \ref{section why}.} of $M$, one can determine the Hilbert series, dimension, multiplicity, projective dimension, and depth of $M$. Furthermore, the Betti numbers provide even finer data than this, and can often be used to detect subtle geometric differences (see Example \ref{Ex:Twisted Cubic} for an obligatory example concerning the twisted cubic curve).

There are many questions one can ask about Betti numbers. What sequences arise as the Betti numbers of some module? Must the sequence be unimodal? How small, or how large, can individual Betti numbers be? How large is the sum?   Questions like this are but just a few examples of those that have been studied in the past decades, and of the flavor we will discuss in this survey.  We will focus on perhaps one of the longest standing open questions in this area, which is due to Buchsbaum--Eisenbud, and independently Horrocks (BEH). Their conjecture proposes a lower bound for each $\beta_i(M)$ depending only on the codimension $c$ of $M$: that $\beta_i(M) \geqslant {c \choose i}$. While the conjecture remains widely open in the general setting, there are some special cases that are known. Moreover, if the conjecture is true, then the total Betti number of $M$, $\beta(M) := \beta_0(M) + \cdots + \beta_n(M)$, must satisfy $\beta(M) \geqslant 2^c$. Recently, Mark Walker \cite{W} proved this bound on the total Betti number --- known as the Total Rank Conjecture --- in all cases except when $\textrm{char } k = 2$. Walker also showed that equality holds if and only if $M$ is isomorphic to $R$ modulo a regular sequence --- such modules are called complete intersections. 

The Betti numbers of modules that are {\it not} complete intersections are quite interesting. For example, it follows from Walker's result that if our module $M$ is not a complete intersection, then $\beta(M) \geqslant 2^c+1$, but there is reason to believe that $\beta(M)$ might be much bigger than $2^c$. Charalambous, Evans, and Miller \cite{CEM} asked if in fact we must have $\beta(M) \geqslant 2^c + 2^{c-1}$, and proved that this holds when $M$ is either a graded module small codimension ($c \leqslant 4$), or a multigraded module of finite length (meaning $c = n$) for arbitrary $c$ \cite{CE,Chara}. More evidence towards this larger bound for Betti numbers has recently been found, including \cite{BS,BoocherW}.

For example, Erman showed \cite{Erman} that if $M$ is a graded module of small regularity (in terms of the degrees of the first syzygies), then not only is the BEH Conjecture \ref{BEH Conjecture} true, but in fact $\beta_i(M) \geqslant \beta_0(M) {c \choose i}$. The first author and Wigglesworth \cite{BoocherW} then extended Erman's work to say that under the same low regularity hypothesis, $\beta(M) \geqslant \beta_0(M) (2^c + 2^{c-1})$. This stronger bound asserts that on average, each Betti number $\beta_i(M)$ is at least 1.5 times $\beta_0(M) {c \choose i}$.

The main goal of this survey is to discuss these lower bounds on Betti numbers and present some of the motivation for these conjectures.  We start with a short introduction to free resolutions and Betti numbers, why we care about them, and some of the very rich history surrounding these topics. We also collect some open questions, discuss some possible approaches, and present examples that explain why certain hypothesis are important. 

\section{What is a Free Resolution?}\label{Section:what is a free resolution}

Let $R = k[x_1, \ldots, x_n]$ be a polynomial ring over a field $k$. We will be primarily concerned with {\bf finitely generated graded $R$-modules} $M$. One important invariant of such a module is the minimal number of elements needed to generate $M$. In fact, this number is the first in a sequence of {\bf Betti numbers} that describe how far $M$ is from being a free module.  Indeed, suppose that $M$ is minimally generated by $\beta_0$ elements; this means there is a surjection from $R^{\beta_0}$ to $M$, say 
$$\xymatrix{
 R^{\beta_0} \ar^-{\pi_0}@{->>}[r] & M.
}$$

If $\pi_0$ is an isomorphism, then $M\cong R^{\beta_0}$ is a {\bf free module} of {\bf rank} $\beta_0$.  Otherwise, it has a nonzero kernel, which will also be finitely generated and can be written as the surjective image of some free module $R^{\beta_1}$:
$$\xymatrix@C=5mm@R=5mm{
R^{\beta_1} \ar@{->>}[rd] \ar@{-->}[rr] & & R^{\beta_0} \ar@{->>}[rr]^-{\pi_0} && M. \\
& \ker(\pi_0) \ar@{^{(}->}[ur]
}$$
Notice that if $M$ is generated by $m_1, \ldots, m_{\beta_0}$, and $\pi_0$ is the map sending each canonical basis element $e_i$ in $R^{\beta_0}$ to $m_i$, then an element $(r_1, \ldots, r_{\beta_0})^T$ in the kernel of $\pi_0$ corresponds precisely to a {\bf relation} among the $m_i$, meaning that 
$$r_1 m_1 + \cdots + r_{\beta_0} m_{\beta_0} = 0.$$
Such relations are called {\bf syzygies}\footnote{Fun fact: in astronomy, a syzygy is an alignment of three or more celestial objects.} of $M$ and the module $\ker{ \pi_0}$ is called the first syzygy module of $M$.  

Continuing this process we can \emph{approximate} $M$ by an exact sequence 
$$\xymatrix{
 \cdots \ar[r] & F_p  \ar[r]^-{\pi_{p}} &  \cdots  \ar[r]^-{\pi_2} &  F_1  \ar[r]^-{\pi_1} & F_0  \ar[r]^-{\pi_0} &  M  \ar[r] & 0
}$$
where each $F_i$ is free. Such an exact sequence is called a {\bf free resolution of $M$}.

If at each step we have chosen $F_i$ to have the minimal number of generators, then we say the resolution is {\bf minimal}, and we set $\beta_i(M)$ to be the rank of $F_i$ in any such minimal free resolution.  This is well-defined, because it is true that two minimal free resolutions of $M$ are isomorphic as complexes.  Furthermore, one has the following, 
$$\beta_i(M) = \rank F_i = \rank_k  \Tor^R_i(M, k).$$  
The {\bf $i$th syzygy module of $M$}, denoted $\Omega_i(M)$, is defined to be the image of $\pi_{i}$, or equivalently the kernel of $\pi_{i-1}$. We note that $\Omega_i(M)$ is defined only up to isomorphism. 

If at some point in the resolution we obtain an injective map of free modules, then its kernel is trivial, and we obtain a finite free resolution, in this case of length $p$:
$$\xymatrix{
 0 \ar[r] & F_p  \ar[r] &  \cdots  \ar[r] &  F_1  \ar[r] & F_0  \ar[r] &  M  \ar[r] & 0.
}$$
If a module $M$ has a finite minimal projective resolution, the length of such a resolution is called the projective dimension of $M$, and we write it $\projdim M$.  
\begin{remark}\label{rmk:Rank-Nullity}
We will often implicitly apply the Rank-Nullity Theorem to conclude that
$$ \beta_i(M) = \rank \Omega_i(M) + \rank \Omega_{i+1}(M).$$
\end{remark}

\begin{example}
	If $M = R/(f_1, \ldots, f_c)$ where the $f_i$ form a regular sequence, then the minimal free resolution of $M$ is given by the {\bf Koszul complex.}   For instance if $c = 4$ then the minimal resolution has the form
	$$\xymatrix{
 0 \ar[r] & R^1  \ar[r]^-1 & R^4 \ar[r]^-3 & R^6  \ar[r]^-3 & R^4  \ar[r]^-1 & R^1  \ar[r]^-0 & M.}$$
 Note that the numbers over the arrows represent the rank of the corresponding map, which is equal to the rank of the corresponding syzygy module $\Omega_i(M)$.
	We will discuss this in more detail in Section \ref{section:building}. We will also see that the ranks occurring in the Koszul complex are conjectured to be the smallest possible for modules of codimension $c$ (see Conjecture \ref{beh ranks syzygies}).
\end{example}
\begin{example}  One of the strongest known bounds on ranks of syzygies is the Syzygy Theorem \ref{The SyzygyTheorem} which states that except for the last syzygy module, the rank of $\Omega_i(M)$ is always at least $i$.   A typical use of such a result might be as follows.  Suppose we had a rank zero module $M$ with Betti numbers $\{1, 7, 8, 8, 7, 1\}$.  Then we could calculate the ranks of the syzygy modules by using Remark \ref{rmk:Rank-Nullity} to obtain the ranks labeled in the diagram below:
$$\xymatrix{
 0 \ar[r] & R^1  \ar[r]^-1 & R^7 \ar[r]^-6 & R^8  \ar[r]^-2 & R^8  \ar[r]^-6 & R^7  \ar[r]^-1 & 
R^1  \ar[r]^-0 & M.}$$
We would also obtain from Remark \ref{rmk:Rank-Nullity} that $\rank \Omega_3(M) = 2$, which we will see violates Theorem \ref{The SyzygyTheorem}. Therefore, such a module does not exist! See also Example \ref{Ex:Pure modules with small Betti numbers two examples 32 and 44}.
\end{example}

\begin{example}
	In \cite{Dugger}, Dugger discusses almost complete intersection ideals and the tantalizing fact that we currently do not know whether or not there is an ideal $I$ of height 5 with minimal free resolution
	$$\xymatrix{
 0 \ar[r] & R^6  \ar[r]^-6 & R^{12} \ar[r]^-6 & R^{10}  \ar[r]^-4 & R^9  \ar[r]^-5 & R^6  \ar[r]^-1 & 
R^1  \ar[r]^-0 & R/I.}$$
\end{example}

\

\noindent
David Hilbert, interested in studying minimal free resolutions as a way to count invariants, was able to prove that finitely generated modules over a polynomial ring always have finite projective dimension  \cite{HilbertSyzygy}.

\begin{theorem}[Hilbert's Syzygy Theorem, 1890]
Let $R = k[x_1, \ldots, x_n]$ be a polynomial ring in $n$ variables over a field $k$. If $M$ is a finitely generated graded $R$-module, then $M$ has a finite free resolution of length at most $n$.
\end{theorem}

While we are primarily interested in studying polynomial rings over fields, Hilbert's Syzygy Theorem is true more generally for any noetherian regular ring. In fact, if we focus our study on local rings instead, the condition that every finitely generated module has finite projective dimension characterizes regular local rings \cite{AuslanderBuchsbaum,Serre}. While we will be working over polynomial rings throughout the rest of the paper, we point out that the theory of (infinite) free resolutions over non-regular rings is quite interesting and rich; \cite{InfiniteGradedFreeRes} and \cite{InfiniteFreeResolutions} are excellent places to start learning about this.

The upshot of Hilbert's Syzygy Theorem is that to each finitely generated $R$-module $M$ we attach a finite list of Betti numbers $\beta_0(M), \, \ldots, \, \beta_n(M)$. Note that while some of these might vanish, $M$ has at most $n+1$ non-zero Betti numbers.

Our main goal in this paper is to discuss the following question: 

\begin{question}\label{main question}
If $M$ is a finitely generated graded module over $R = k[x_1,\ldots,x_n]$, where $k$ is a field, can we bound the Betti numbers of $M$, either from above or below?
\end{question}

As we will see, there are many results and conjectures relevant to the answer to this question.  Feel free to skip the next section if you can't handle the suspense!

\section{Why Study Resolutions?}\label{section why}

Before getting to the heart of the matter in Section \ref{section BEH conjecture}, we would first like to offer some motivation as to why one might care about Betti numbers at all.  

\subsection{Betti Numbers Encode Geometry}

In a sense, a minimal free resolution of $M$ contains redundant information --- after all, the first map $\pi_1\!\!: F_1 \to F_0$ is a presentation of $M$. However, suppose we do not know the maps in the resolution, but just the \textbf{numerical data} of the resolution, namely the numbers $\{\beta_i\}$.  Surprisingly, this coarse invariant encodes much geometric and algebraic information about $M$. First of all, the Betti numbers $\beta_i$ tell us that $M$ has $\beta_0$ generators, that there are $\beta_1$ relations among those generators, and $\beta_2$ relations among those relations, and so on. But the Betti numbers also encode more sophisticated information about $M$.  For instance, since rank is additive across exact sequences, we have
$$\rank M = \beta_0 - \beta_1 +  \cdots + (-1)^n \beta_n.$$  

Moreover, if we have a graded module $M$, we can take the resolution of $M$ to be a graded resolution, and if among the $\beta_i$ generators of $\Omega_i(M)$, exactly $\beta_{ij}$ of them live in degree $j$, then the following formula gives the Hilbert series for $M$:
\begin{equation}\label{equation for HS}HS(M) = \displaystyle\frac{\,\, \displaystyle\sum_{i=0}^d (-1)^i \beta_{ij} t^j \,\,}{(1-t)^d}.\end{equation}
We recall that the {\bf Hilbert series} of $M$ is a power series that encodes the $k$-vector space dimension of each graded piece $M_i$ of $M$, as follows:
$$HS(M) = \displaystyle\sum_{i=0}^\infty \dim_k (M_i) t^i.$$
This is a classical tool that contains important algebraic and geometric information about our module. For example, once we write $HS(M) = p(t)/(1-t)^m$ with $p(1) \neq 0$, we have $\dim (M) = m$ and $p(1)$ is equal to the degree of $M$. So just by knowing its (graded) Betti numbers, we can then determine the multiplicity (i.e. degree), dimension, projective dimension, Cohen-Macaulayness, and other properties and invariants of a module $M$.  
 
The following example gives the spirit of these ideas:

\begin{example}\label{Ex:3 lines in k^3}
Suppose that $R = k[x,y,z]$ and that $M = R/(xy,xz,yz)$ corresponds to the affine variety defining the union of the {\color{orange} three} coordinate {\color{brown}lines} in $k^3$. This variety has dimension {\color{brown}one} and degree {\color{orange} three}. Let us illustrate how the (graded) Betti numbers communicate this. The minimal free resolution for $M$ is
\[
\xymatrix{   
0\ar[r] & R^2 \ar[rrr]^{\displaystyle{\psi} = \begin{bmatrix} z & 0 \\ -y & y \\ 0 & -x \end{bmatrix}
  } &&& R^3 \ar[rrrr]^{\displaystyle{\phi} = \begin{bmatrix} xy& xz & yz\end{bmatrix}} &&&& R \ar[r] & M.
}
\]
From this minimal resolution, we can read the Betti numbers of $M$:
\begin{itemize}
\item $\beta_0 = 1$, since $M$ is a cyclic module;
\item $\beta_1 = {\color{cyan}3}$, and these three quadratic generators live in degree {\color{ForestGreen}2};
\item $\beta_2 = \textcolor{magenta}2$, and these represent linear (degree 1) syzygies on quadrics (degree 2), and thus live in degree {\color{blue} 3  $(=1+2)$ }.
\end{itemize}

We can include this \emph{graded} information in our resolution, and write a \emph{graded} free resolution of $M$:
\[
\xymatrix{   
0\ar[r] & R({\color{blue} -3})^{\textcolor{magenta}2} \ar[rrr]^-{\displaystyle{\psi} = \begin{bmatrix} z & 0 \\ -y & y \\ 0 & -x \end{bmatrix}
  } &&& R({\color{ForestGreen} -2})^{\textcolor{cyan}3} \ar[rrrr]^-{\displaystyle{\phi} = \begin{bmatrix} xy& xz & yz\end{bmatrix}} &&&& R \ar[r] & M.
}
\]

The $R({\color{ForestGreen} -2})^{\textcolor{cyan} 3}$ indicates that we have \textcolor{cyan}{three} generators of degree {\textcolor{ForestGreen} 2}. Formally, the $R$-module $R(-a)$ is one copy of $R$ whose elements have their degrees shifted by $a$: the polynomial $1$ lives in degree $0$ in $R$ and degree $a$ in $R(-a)$, and in general the degree $d$ piece of $R(-a)$ consists of the elements of $R$ of degree $d-a$. With this convention, the map $\phi$ keeps degrees unchanged --- we say it is a degree $0$ map: for example, it takes the vector $[1,1,1]^T$, which lives in degree $2$, to the element $xy + xz + yz$, which is an element of degree $2$. When we move on to the next map, $\psi$, we only need to shift the degree of each generator by {\color{orange}1}, but since $\psi$ now lands on $R({\color{ForestGreen}-2})^3$, we write $R({\color{blue} -3})^2$. 

The graded Betti number $\beta_{ij}(M)$ of $M$ counts the number of copies of $R(-j)$ in homological degree $i$ in our resolution. So we have
$$\beta_{00}=1, \beta_{1{\textcolor{ForestGreen} 2}} = {\textcolor{cyan} 3}, \textrm{and } \beta_{2{\textcolor{blue} 3}}= {\textcolor{magenta}2}.$$

We can collect the graded Betti numbers of $M$ in what is called a \emph{Betti table}:

$$
\begin{array}{cc}
\begin{array}{r|ccccc} 
\beta(M) & 0 &1&2&\\ \hline
0& \beta_{00}  & \beta_{11}  & \beta_{22}  & \\
1 & \beta_{01}  & \beta_{12}  & \beta_{23}  &  \\
& 
\end{array},
& \begin{array}{r|ccccc} 
\beta(M) & 0 &1&2&\\ \hline
0 & 1 & - & - & \\
1 & - & {\textcolor{cyan} 3} & {\textcolor{magenta}2} &  \\
& 
\end{array}.
\end{array}
$$
\begin{remark}
	To the reader who is seeing Betti tables for the first time, we point out that although we will write resolutions so that the maps go from left to right, and thus the Betti numbers appear from right to left $\{\ldots, \beta_2, \beta_1, \beta_0 \}$  in a Betti table, the opposite order is used.  Furthermore, by convention, the entry corresponding to $(i,j)$ in the Betti table of $M$ is $\beta_{i,i+j}(M)$, and {\it not} $\beta_{ij}(M)$.
\end{remark}

Finally, we can use this information to calculate the Hilbert series of $M$:
$$HS(M) = \frac{1t^0 - {\textcolor{cyan} 3}t^{\textcolor{ForestGreen}2} + {\textcolor{magenta}2}t^{\textcolor{blue} 3}}{(1-t)^3} = \frac{1+2t}{(1-t)^1},$$
and since this last fraction is in lowest terms, we see that the dimension of $M$ is ${\color{brown}1}$ (the degree of the denominator) and that the degree of $M$ is equal to $p(1)=1+2 \cdot 1 = {\color{orange} 3}$.  Recall that $M$ corresponded to the union of 3 lines.  Notice that in this example, the projective dimension of $M$ is $2$, which is equal to the codimension $3-{\color{brown} 1} = 2$ of $M$. Hence, $M$ is Cohen-Macaulay. In summary, we can get lots of information about $M$ from its (graded) Betti numbers.
\end{example}

\begin{example}(The Hilbert series doesn't determine the Betti numbers) 
Let $k$ be a field, $R = k[x,y]$, and consider the two ideals 
$$I = (x^2,xy,y^3) \quad \textrm{ and } \quad J = (x^2,xy+y^2).$$
One can check that both $R/I$ and $R/J$ have the same Hilbert series:
$$HS(R/I) = HS(R/J) = 1 + 2t + 1t^2.$$
However, these modules have different Betti numbers. We work out the minimal free resolution and Betti numbers for $R/I$. Since $I$ has two generators of degree 2 and one of degree 3, there are graded Betti numbers $\beta_{12}$ and $\beta_{13}$. Similarly, the two minimal syzygies of $R/I$ correspond to the relations
\begin{align*}
	y(x^2) - x(xy) = 0 & \mbox{ \ which has degree 3, so } \beta_{23} = 1\\
	& \textrm{ and } \\ 
	y^2(xy) - x(y^3) = 0 & \mbox{\ which has degree 4, so } \beta_{24} = 1.
\end{align*}

Continuing this process, we find the following minimal free resolutions and graded Betti numbers for $R/I$ and $R/J$, respectively:

\noindent
\begin{minipage}{0.52\textwidth}
\vspace{0.2em}
\[\resizebox{\displaywidth}{!}{
\xymatrix@C=2mm@R=3mm{   
 {\begin{array}{c} 
R({\color{blue} -3})^{\textcolor{magenta}1} \\
\bigoplus \\
R({\color{blue} -4})^{\textcolor{magenta}1} \\
\end{array}}
\ar[rr]^-{ \begin{bmatrix} y & 0 \\ -x & y^2 \\ 0 & -x \end{bmatrix}
  } &&
  {\begin{array}{c}
  R({\color{ForestGreen} -2})^{\textcolor{cyan}2} \\ 
  \bigoplus \\ 
  R({\color{ForestGreen} -3})^{\textcolor{cyan}1}
  \end{array}}
  \ar[rr]^-{\begin{bmatrix} x^2& xy & y^3\end{bmatrix}} && R
  \\    &  {\begin{array}{c} \beta_{2{\color{blue}3}}(R/I) = {\textcolor{magenta}1}\\ \beta_{2{\color{blue}4}}(R/I) = {\textcolor{magenta}1} \end{array}}& &
  {\begin{array}{c} \beta_{1{\color{ForestGreen}2}}(R/I) = {\textcolor{cyan}2}\\ \beta_{1{\color{ForestGreen}3}}(R/I) = {\textcolor{cyan}1} \end{array}}
 }}
\]
\end{minipage}
\vline
\begin{minipage}{0.48\textwidth}
\[\resizebox{\displaywidth}{!}{
\xymatrix@C=2mm@R=7mm{   
R({\color{blue} -4})^{\textcolor{magenta}1} 
\ar[rr]^-{\begin{bmatrix} xy+y^2  \\ -x^2 \end{bmatrix}
  } &&
   R({\color{ForestGreen} -2})^{\textcolor{cyan}2} 
  \ar[rr]^-{\begin{bmatrix} x^2& xy+y^2 \end{bmatrix}} && R.
  \\ & {\beta_{2{\color{blue}4}}(R/J) = {\textcolor{magenta}1}} &&
  {\beta_{1{\color{ForestGreen}2}}(R/J) = {\textcolor{cyan}2}}
  }}
  \]
  \end{minipage}
  \vspace{1em}
$$\begin{array}{cc}
\begin{array}{r|ccccc} 
\beta(R/I) & 0 &1&2&\\ \hline
0& 1  & -  & -  & \\
1 & -  & 2  & \textcolor{red}{1}  &  \\
2 & -  & \textcolor{red}{1} & 1  &  \\ 
\end{array}  
\hspace{4em} & \hspace{6em}
\begin{array}{r|ccccc} 
\beta(R/J) & 0 &1&2&\\ \hline
0& 1  & -  & -  & \\
1 & -  & 2  & -  &  \\
2 & -  & -  & 1  &  \\ 
\end{array}
\end{array}$$  

\
  
Finally, if we calculate the Hilbert series from Equation \ref{equation for HS}, we notice that the calculation is the same for $R/I$ and $R/J$:
$$HS(R/I) = \frac{1 - 2t^2 \textcolor{red}{- t^3 + t^3} + t^4}{(1-t)^3} =  \frac{1 - 2t^2 + t^4}{(1-t)^3} = HS(R/J).$$ The cancellation of the $\textcolor{red}{t^3}$ terms is known as a \textbf{consecutive cancellation}, and one can see the two $\textcolor{red}{1} $s on the diagonal in the Betti table for $R/I$.   For the reader who knows about Gr\"obner degenerations, $I$ is the initial ideal of $J$ coming from a Lex term-order.  Any such degeneration will preserve the Hilbert series, but not necessarily the Betti numbers.  For results concerning the relationship between the Betti numbers of ideals and those of their initial ideals, see \cite{Ardila,B,ConcaVarbaro,CDNG2,CDNG,Moh}. \end{example}

\begin{example}\label{Ex:Twisted Cubic}
We would be remiss if, in this article dedicated to David Eisenbud on his birthday, we didn't also mention that the connection between graded Betti numbers and geometry is a rich and beautiful story.  In his book \cite{Eisenbud}, he paints a story that begins with the following surprising fact from geometry. If $X$ is a set consisting of seven general\footnote{this means that no more than 3 lie on a plane and no more than 5 on a conic.} points in $\mathbb{P}^3$, then the Hilbert series of the coordinate ring for $X$ is completely determined by this data.  However, this is not sufficient to determine the Betti numbers of the coordinate ring of $X$.  Indeed, these numbers are either $\{1, 4, 6, 3\}$ or  $\{1,6, 8, 3 \}$
depending on whether or not the points lie on a curve of degree 3.  
\end{example}

\subsection{Resolutions for Ideals with Few Generators}\label{section:building}
Over a polynomial ring $R = k[x_1, \ldots, x_n]$, calculating a free resolution is
tantamount to producing the sets of dependence relations among the generators of a module.  In simple cases this is straightforward, as the following example shows:

\begin{example}\label{Ex:principal ideal  M = R/f}
Consider the module $M = R/(f)$, where $f$ is a homogeneous polynomial in $R$. Then 
$$\xymatrix{0 \ar[r]&  R \ar[r]^{\begin{bmatrix}f\end{bmatrix}} & R\ar[r] & M}$$
is a minimal free resolution of length $1$, since over our polynomial ring $R$, $f$ is a \emph{regular} element and cannot be killed by multiplication by any nonzero element. 

If $I$ is an ideal minimally generated by two polynomials $f$ and $g$, then the minimal free resolution of $R/I$ has length two. Indeed, if $c = \gcd(f,g)$, then the following is a minimal free resolution: 
$$\xymatrix{0 \ar[r]&  R \ar[rr]^{\begin{bmatrix} g/c\\ -f/c \end{bmatrix}} && R^2\ar[rr]^{\begin{bmatrix} f & g \end{bmatrix}} && R \ar[r] & R/I}.$$
\end{example}
This example  can be summarized by the following result:
\begin{proposition} If $I$ is an ideal in a polynomial ring $R$ that is minimally generated by one or two homogeneous polynomials, then the projective dimension of $R/I$ is equal to the minimal number of generators, and the Betti numbers are either $\{1,1\}$ or $\{1,2,1\}$. 
\end{proposition}

Whatever optimistic generalization of this proposition one might have in mind for ideals with $3$ or more generators will certainly fail to be true, as we have the following astonishing results of Burch and Bruns:

\begin{theorem}[Burch, 1968 \cite{Burch}]
For each $N \geqslant 2$, there exists a three-generated ideal $I$ in a polynomial ring $R = k[x_1, \ldots, x_N]$ such that $\projdim (R/I) = N$.
\end{theorem}

So we can always find free resolutions of maximal length by simply using $3$ generated ideals. In fact, in some sense ``every'' free resolution is the free resolution of a $3$-generated ideal:

\begin{theorem}[Bruns, 1976 \cite{Bruns}]\label{Bruns3Generated}
	Let $R = k[x_1, \ldots, x_n]$ and
	$$\xymatrix{0 \ar[r] & F_n \ar[r] & F_{d-1} \ar[r] & \cdots \ar[r] & F_2 \ar[r] & F_1 \ar[r] & F_0 \ar[r] & M}$$
	be a minimal free resolution of a finitely generated graded $R$-module $M$. Then there exists a $3$-generated ideal $I$ in $R$ with minimal free resolution
	$$\xymatrix{0 \ar[r] & F_n \ar[r] & \cdots \ar[r] & F_3 \ar[r] & F_2' \ar[r] & R^3 \ar[r] & R \ar[r] & R/I}.$$
\end{theorem}
\begin{remark}
Note that the rank of $F_2'$ may be different than that of $F_2$, but a rank calculation yields that 
$$\rank F_2' = 3 - 1 + \rank F_3 - \rank F_4  + \cdots \pm \rank F_n = 2 + \rank F_2 - \rank F_1 + \rank F_0.$$
From this, it follows that $\beta_2$ can be arbitrarily large for 3-generated ideals.
\end{remark}

Our point in presenting these results is to make plain that free resolutions are complicated --- even for ideals with 3 generators!  However, if in Example \ref{Ex:principal ideal  M = R/f} we add a further restriction for the ideal $I=(f,g)$ and require that $f$ and $g$ have no common factors (meaning that $g$ is a \emph{regular} element modulo $f$), then the only relations between $f$ and $g$ are given by the ``obvious'' relation that $gf - fg =0$.  This fact \emph{does} generalize nicely to any set $\{f_1, \ldots, f_c\}$ of homogeneous polynomials provided $f_i$ is a regular element modulo the previous $f_j$.  Such elements form what is called a {\bf regular sequence}, and the ideal they generate is resolved by the {\bf Koszul complex}. Rather than introducing the topic here, we point the reader to some of the many nice references for learning about the Koszul complex, such as \cite[Chapter 17]{Eisenbud}, \cite[Section 1.6]{BrunsHerzog}, or \cite[Example 1.1.1]{InfiniteFreeResolutions}.

The most important fact we will need about the Koszul complex is that it is a resolution (of $R/(f_1, \ldots, f_c)$) if and only if the $f_1, \ldots, f_c$ form a regular sequence, and that the Betti numbers (and ranks of syzygy modules) of the Koszul complex are given by binomial coefficients.  

\begin{theorem}\label{Theorem: Koszul Ranks}
If $I$ is an ideal generated by a regular sequence of $c$ homogeneous polynomials, then 
$$\rank \Omega_i(R/I) = {c-1 \choose i-1},$$
and therefore
$$\beta_{i}(R/I) = { c \choose i}.$$
\end{theorem}

\begin{remark}
To the reader not familiar with Koszul complexes, it might be instructive to carefully write out the maps involved to get a feel for how resolutions are constructed.  Essentially, the point is that the generating $i$th syzygies are built from using $i$ generators and the fact that $f_jf_i = f_if_j$.  Alternatively, perhaps the quickest way to define the Koszul complex is just to take the tensor product of the $c$ minimal free resolutions of $R/(f_i)$:
\[
\xymatrix{
0 \ar[r] &  R \ar[r]^{f_i} & R \ar[r] & 0.
}
\]
Since multiplication by $f_i$ has rank one, if one calculates the ranks in the tensor product inductively, one will see Pascal's Triangle appearing, providing a justification of the claims in Theorem \ref{Theorem: Koszul Ranks}.
\end{remark}

\subsection{How Small Can the Ranks of Syzygies Be?} 
If $I$ is an ideal that is generated by a regular sequence then as we saw in the previous section, the minimal free resolution for $R/I$ is given by the Koszul complex. For instance, if $I$ has height $8$, then $\beta_4(R/I)$ will be equal to ${8 \choose 4} = 70$,  and the syzygy module $\Omega_4(R/I)$ will have rank ${7 \choose 3} = 35$.   We will see in the next section (Conjectures \ref{BEH Conjecture} and \ref{beh ranks syzygies}) that among all ideals of height $8$ these numbers are conjectured to be the smallest possible values for $\beta_4$ and $\rank \Omega_4$ respectively. In short, these conjectures assert a relationship between the ranks of syzygies and the height (or codimension) of the ideal. Before we present these conjectures, which will occupy the remainder of the paper, we close with an example and theorem that give the sharpest possible bound for ranks of syzygies if one does not refer to codimension.  

\begin{example}[Bruns, 1976 \cite{Bruns}] \label{Ex:Bruns 2i+1} Let $R = k[x_1, \ldots, x_n]$.  There is a finitely generated module $M$ over $R$ with the following resolution: 
$$\xymatrix{0 \ar[r] & R \ar[r] & R^n \ar[r] & R^{2n-3} \ar[r] & \cdots \ar[r] & R^5 \ar[r] & R^3 \ar[r] & R \ar[r] & M \ar[r] & 0.}$$
In other words, the $i$th Betti number is $2i+1$ except for the last two Betti numbers.  This is the case for an even nicer reason: if one calculates the ranks of each syzygy module (which can be read off as the rank of the $i$th map $\pi_i$ in the resolution) one sees that the ranks are: 
$$\xymatrix{0 \ar[r] & R \ar[r]^-{1} & R^n \ar[r]^-{n-1} & R^{2n-3} \ar[r]^-{n-2} & \cdots \ar[r]^-{3} & R^5 \ar[r]^-{2} & R^3 \ar[r]^-{1} & R \ar[r]^-{0} & M \ar[r] & 0.}$$
In other words, in this example the $i$th syzygy module has rank equal to $i$, except for the last one.   This bound holds for any module, which is the content of the great Syzygy Theorem.
\end{example}

\begin{theorem}[Syzygy Theorem, Evans--Griffith, 1981 \cite{EG}]\label{The SyzygyTheorem}
Let $M$ be a finitely generated module over a polynomial ring $R$. If $\Omega_i(M)$ is not free, then $\rank \Omega \geqslant i$. Hence, if $\projdim M = p$, then 
$$\rank \Omega_i(M) \geqslant i, \ \mbox{for $i < p$.}$$
Moreover,
$$\beta_i(M) = \rank \Omega_i(M) + \rank \Omega_{i+1}(M) \geqslant \begin{cases} 
2i + 1 & \mbox{if $i < p - 1$}\\
p & \mbox{if $i =  p - 1$}\\
1 & \mbox{if $i =  p$}
\end{cases}
$$
where $\Omega_i(M)$ denotes the $i$th syzygy module of $M$.
\end{theorem}

The Syzygy Theorem together with Bruns' example provides a sharp lower bound for $\beta_i(M)$.  Without further conditions on $M$, there is not much more we can say. However, if we add additional hypotheses on $M$ --- for instance, requiring $M$ to be Cohen-Macaulay, or of a fixed codimension $c$ --- then the bounds above appear to be far from sharp. Indeed, we will discuss a conjecture that states that in fact $\beta_i(M) \geqslant {c \choose i}$; when $c$ is large, this conjecture is much stronger than the Syzygy Theorem's bound of $2i + 1$. Note that the ideal in Example \ref{Ex:Bruns 2i+1} is of codimension $2$.

\subsection{Other Possible Directions}
Before we begin to focus on codimension, we want to say that there are many distinct and interesting alternative questions on bounds for Betti numbers that have been considered.  We present some possibilities below.

One could decide to study ideals and then fix the number of generators of $I$; for example, one could study the sets of Betti numbers of ideals defined by $5$ homogeneous polynomials. Theorem \ref{Bruns3Generated} shows that this approach will not allow for any upper bounds, except in trivial cases.  

Refining this idea, one could add a condition on the \emph{degrees} of the generators of these ideals, and for example ask what the maximal Betti numbers for an ideal with $3$ quadratic generators might be.  This question is tractable, though incredibly difficult.  Note that here we are not saying how many variables are in the ring $R$.  For instance, the largest Betti numbers possible for an ideal generated by $3$ quadrics is $\{1, 3, 5, 4, 1\}$; note that the projective dimension is $4$.  More generally, the question of whether there exists an upper bound on the projective dimension of an ideal defined by $r$ forms of degree $d_1, \ldots, d_r$ depending only on $r$ and $d_1, \ldots, d_r$, and not on the number of variables, is known as Stillman's Conjecture, and has been solved by Ananyan and Hochster \cite{StillmansConjecture} in general. The question of providing effective upper bounds is much harder, and some of the efforts in this direction can be found in \cite{BigPolynomialRings}. See \cite{StillmansConjectureBulletin} for an exposition on some of the followup results that expanded on the ideas initiated by Ananyan and Hochster in their proof of Stillman's conjecture; see also \cite{StillmansConjectureSurvey} for a survey and \cite{CraigPaoloJason,CavigliaLiang} for related work on the subject. 

We saw in sections \ref{Section:what is a free resolution} and \ref{section why} that the (graded) Betti number determine the Hilbert series; however, there can be many distinct sets $\{\beta_{ij}(M)\}$ for $R$-modules $M$ all with the same given Hilbert series.  If one fixes a Hilbert series, what are the possible sets $\{\beta_{ij}(M)\}$ for modules $M$ with Hilbert series $h(t)$?   The following theorems give a beautiful answer that provides an upper bound for the Betti numbers.

\begin{theorem}[Bigatti, 1993 \cite{Bigatti}, Hulett, 1993 \cite{Hulett}, Pardue, 1996 \cite{Pardue}]
Let $I$ be a homogeneous ideal in $R=k[x_1, \ldots, x_n]$.  Consider the set 
$$\mathcal{H} = \{ J \subseteq R \textrm{ an ideal } \ | \ HS(R/J) = HS(R/I)\}.$$
There exists an ideal $L\in \mathcal{H}$ with the property that among all ideals in $H$, the Betti numbers of $L$ are the largest: 
$$\beta_{ij}(R/J) \leqslant \beta_{ij}(R/L) \ \mbox{for all $i,j$ and for all $J\in \mathcal{H}$.}$$
\end{theorem}

The ideal $L$ that achieves the largest Betti numbers in the Theorem can be described explicitly, and goes back 100 years to work of Macaulay \cite{Macaulay}; it is the known as the {\bf Lex-segment ideal}. To construct $L$, we start by going over each degree $D$ and ordering all the monomials in $R_D$ lexicographically. Then we collect the first $\dim_k (J_D)$ monomials in degree $D$, for all $D$. Macaulay showed the ideal $L$ generated by all these monomials has the same Hilbert function as our original ideal $J$; in other words, it is an ideal in $\mathcal{H}$. Bigatti, Hulett, and Pardue then showed that this special ideal has in fact the largest possible Betti numbers with the same Hilbert function as $I$. Moreover, if we fix a Hilbert polynomial, and consider all the saturated ideals $I$ with that fixed Hilbert polynomial, there is also a particular lex-segment ideal that maximizes the Betti numbers \cite{LargestBettiPolynomial}.

Finally, we remark that while this paper is devoted to the ranks of modules appearing in a minimal resolution --- that is, the study of acyclic complexes.  There has been much work devoted more generally to complexes, or even more generally to differential modules. For instance, it was conjectured in \cite[Conjecture 5.3]{AvBuIy} that if $F_\bullet$ is any complex over a $d$-dimensional local ring, and if the homology $H(F)$ has finite length, then $\sum_i \rank F_i \geqslant 2^d$. This was shown in \cite{AvBuIy} for the case when $d \leqslant 3$, and in \cite{BoocherDevries} in the multigraded setting (for all $d$).  However, the conjecture is false in general.  Indeed, in \cite{IyWa}, an example is given of a complex of $R$-modules such that $H(F)$ has length $2$ but $\sum_i \rank F_i < 2^d$ for all $d \geqslant 8$. See also \cite{Carlsson3,Carlsson2,Carlsson1,BrownErman}.

\vspace{1em}
In the remainder of the paper we will state several conjectures concerning lower bounds for the $\beta_i(R/I)$ in terms of $c = \codim R/I$. As an appetizer, notice that the Krull altitude theorem asserts that the codimension of $R/I$ must be at most the minimal number of generators, i.e.  $\beta_1(R/I) \geqslant c.$ Meanwhile, the Auslander-Buchsbaum formula above guarantees that the length of the resolution of $R/I$ is at least the codimension $c$, which implies that $\beta_c(R/I) \geqslant 1$.  With these two classical results giving us information about Betti numbers in terms of codimension, we now proceed to the main conjecture we want to focus on.


\section{The Buchsbaum--Eisenbud--Horrocks Conjecture and the Total Rank Conjecture}\label{section BEH conjecture}

\noindent In the late 1970s, Buchsbaum and Eisenbud \cite{BE}, and independently Horrocks \cite[Problem 24]{HartshorneProblems}, conjectured that the Koszul complex is the smallest free resolution possible; more precisely, that the Betti numbers of any finitely generated module are at least as large as those of a complete intersection of the same codimension as given in Theorem \ref{Theorem: Koszul Ranks}:

\begin{conjecture}[BEH Conjecture]\label{BEH Conjecture}
	Let $R = k[x_1, \ldots, x_n]$, where $k$ is a field, and $M$ be a nonzero finitely generated graded $R$-module of codimension $c$, meaning that $\het \textrm{ann}(M) = c$. Then
	$$\beta_i(M) \geqslant {c \choose i}$$
	for all $0 \leqslant i \leqslant \projdim_R M$.  
\end{conjecture}

Actually, both Buchsbaum and Eisenbud \cite{BE} and Horrocks \cite[Problem 24]{HartshorneProblems} propose the following stronger conjecture:

\begin{conjecture}[Stronger BEH Conjecture for the ranks of the syzygies]\label{beh ranks syzygies}
	Let $R = k[x_1, \ldots, x_n]$, where $k$ is a field, and $M$ be a nonzero finitely generated graded $R$-module of codimension $c$. Then
	$$\rank(\Omega_i(M)) \geqslant {c - 1 \choose i-1}.$$
\end{conjecture}
Originally, Horrocks' problem was stated for finite length modules over a regular local ring, i.e., the case that $\codim M$ was as large as possible, and equal to the dimension of the ring. On the other hand, Buchsbaum and Eisenbud were interested in resolutions of $R/I$ for a general local ring $R$. They conjectured that the minimal free resolution of $R/I$ possessed the structure of a commutative associative differential graded algebra; they then showed that if this held, and $I$ had grade $c$, then the corresponding inequalities (which they independently attribute to J\"urgen Herzog) on the ranks above would hold:

\begin{theorem}[Buchsbaum--Eisenbud, 1977, Proposition 1.4 in \cite{BE}] 
If $R/I$ has codimension $c$ and the minimal free resolution of $R/I$ possesses the structure of an associative commutative differential graded algebra, then $\beta_i(R/I) \geqslant {c \choose i}$ for all $i$.  Furthermore, the rank of the $i$th syzygy module is at least ${c - 1 \choose i-1}$.
\end{theorem}

For some time it was open whether or not all resolutions could be given such a DGA structure. It turns out that this is not necessarily the case \cite[Example 5.2.2]{LuchoObstructions}, though notably any algebra $R/I$ of projective dimension at most $3$ or of projective dimension $4$ that is Gorenstein will have such a resolution \cite{BE,kustinmillerclass,kustinchar2}. See also \cite{AKM} for more on the $\projdim(R/I) \leqslant 3$ case.

\begin{remark}
Throughout, we will adopt the convention that ${n \choose k}$ is zero unless $0 \leqslant k \leqslant n$. 
\end{remark}

As a motivating example, let $R/I$ be a cyclic module of codimension $c$.  
\begin{itemize}
\item The principal ideal theorem guarantees that $I$ must be generated by at least $c$ elements, so $\beta_1(R/I) \geqslant {c\choose 1}$.  
\item The Auslander--Buchsbaum formula implies that $\projdim (R/I) \geqslant c$, which implies that $\beta_c(R/I) \geqslant {c\choose c}.$ 
\item If $I$ is generated by exactly $c$ elements, then $R/I$ is resolved by the Koszul complex, and then $\beta_i(R/I) = { c \choose i }$ for all $i$. 
\end{itemize}

If $I$ has more than $c$ generators, then $I$ will not be a complete intersection, and in general there is no structural result concerning its minimal free resolution.  However, it stands to reason (at least for optimists) that perhaps the Betti numbers can only increase as the number of generators grows and grows.

    In the rest of this paper we have two goals. First, we want to survey various generalizations of the BEH Conjecture and give the state of the art for each of these. Second, we want to include a few basic constructions and techniques that could be helpful to those who want to work in this field.  For a more thorough treatment, we refer the reader to the book \cite{SyzygiesBook} and survey article \cite{SundanceCollection}. 
    
    We have opted to give a summary of classical results on the BEH Conjecture first, but we want to point out right away that an immediate consequence of the BEH Conjecture is that if the conjecture is true, then the sum of the Betti numbers will be at least $2^c$.  This weaker conjecture, known as the Total Rank Conjecture, was proven by Walker in 2018 \cite{W}. Since then, there has been increasing evidence that apart from complete intersections, which are resolved by the Koszul complex, it may be true that in fact the sum of the Betti numbers is always at least $2^c  + 2^{c-1}$.  In the final section of the survey, we present the case for this stronger conjecture.

\subsection{General Purpose Tools} 
The BEH Conjecture is known in a surprisingly small number of cases. Indeed, as a first challenge, it is open an open question whether $\beta_2(R/I) \geqslant {5 \choose 2}$ whenever $I$ is an ideal of codimension $5$.  In this section, we present a collection of general purpose tools and use them to show that if $c \leqslant 4$ then the conjecture holds.  We also carefully describe how localization can reduce the conjecture to the finite length case, provided we work over arbitrary regular local rings. 

\begin{proposition}[Buchsbaum--Eisenbud, 1973, Theorem 2.1 in \cite{BERemarks}, see also \cite{Macaulay}]\label{prop: generalized PIT}
Suppose that $M$ is a module of codimension $c$.  Then 
$$\beta_1(M) - \beta_0(M) + 1 \geqslant c.$$
If equality holds, then $M$ is resolved by the Buchsbaum--Rim complex.
\end{proposition}

Note that this result includes both the Principal Ideal Theorem (when $M = R/I$ and thus $\beta_0(M) =1$) and the fact that the Koszul complex (a special instance of the Buchsbaum--Rim complex \cite{BuchsbaumRim1963}) resolves complete intersections. Below is a version of this result in terms of Betti numbers:  

\begin{corollary}\label{Cor Numerical Generalized PIT}
If $M$ is a module of codimension $c$, then 
$$\beta_1(M) \geqslant \beta_0(M) + c - 1.$$
If equality holds, then for all $i \geqslant 2$
$$\beta_i(M) = {\beta_0(M) + i -3 \choose i-2}{\beta_1(M) \choose \beta_0(M) + i - 1}.$$
\end{corollary}
As an exercise, the reader can prove that if $\beta_1(M) = \beta_0(M) + c - 1$ then the BEH conjecture holds, by the equality of binomial coefficients above.

Discounting cases when equality holds, this lower bound $\beta_1(M)> \beta_0(M) +c -1$ might not at first glance seem very useful, since it only gives information about $\beta_1(M)$.  However, when $M$ is Cohen-Macaulay we can use this result to also gain information about $\beta_{c-1}(M)$ as well by appealing to duality.  Indeed, if $M$ is a Cohen-Macaulay module, meaning that the codimension $c$ of $M$ is equal to its projective dimension, then applying $\Hom( - , R)$ to a resolution of $M$ will yield a resolution of $\Ext^c_R(M,R)$, another Cohen-Macaulay module. This yields the following observation:

\begin{proposition}\label{Prop about reverse of Betti sequence}
If $\{\beta_0, \ldots, \beta_c\}$ is the Betti sequence for a Cohen-Macaulay module, then so is the reverse sequence $\{\beta_c,\ldots, \beta_0\}$.
\end{proposition}

As an application of these ideas, let us use these results to prove Conjecture \ref{BEH Conjecture} for $c \leqslant 4$. We focus on $c = 3$ and $c=4$, as the smaller cases follow immediately from the principal ideal theorem.

When $\mathbf{c =3}$, Corollary \ref{Cor Numerical Generalized PIT} and Proposition \ref{Prop about reverse of Betti sequence} imply that 
$$\{\beta_0, \beta_1, \beta_2,\beta_3\} \geqslant \{ \beta_0, 3 + \beta_0 - 1, 3 + \beta_3 -1 , \beta_3 \} \geqslant \{ 1, 3, 3, 1\}$$
where the inequalities are interpreted entry by entry.

Similarly, for $\mathbf{c = 4}$ we obtain
$$\{\beta_0, \beta_1, \beta_2,\beta_3, \beta_4 \} \geqslant \{ \beta_0, 4 + \beta_0 - 1, \beta_2, 4 + \beta_4 -1 , \beta_4\} \geqslant \{ 1, 4, \beta_2, 4 , 1\}.$$

From here, we can apply the Syzygy Theorem (\ref{The SyzygyTheorem}) and notice that in a minimal free resolution
$$\xymatrix{
 0 \ar[r] & R^{\beta_4} \ar[r] & R^{\beta_3} \ar[r]^{\pi_3} & R^{\beta_2} \ar[r]^{\pi_2} & R^{\beta_1} \ar[r] & R^{\beta_0} \ar[r] & M,
}
$$
the image of $\pi_3$ is equal to $\Omega_3(M)$, and thus the rank of $\pi_3$ is at least $3$ by the Syzygy Theorem; here we used that $c = 4$, so that $\Omega_3(M)$ is not free.

Similarly, working now on the resolution of the dual $\Ext^4_R(M,R)$, we can see that the rank of $\pi_2$ must be at least $3$ as well. Hence 
$$\beta_2 = \rank \pi_2 + \rank \pi_3 \geqslant 6,$$
as required.

However, if we try the same tricks with $\mathbf{c =5}$, the best we can get is that 
$$
\{\beta_0  ,\beta_1  ,\beta_2  ,\beta_3  ,\beta_4  ,\beta_5 \}  \geqslant \{1, 5, 7, 7, 5, 1  \} .
$$
There are, however, other techniques one could use to try and complete this case:

\begin{itemize}
\item Suppose $M$ is cyclic, that is, $\beta_0 = 1$.  Then one may assume that $\beta_1 > 5$.  Indeed, if $\beta_1 = 5$, then $M$ is a complete intersection and the Koszul complex is a resolution. Surprisingly, Conjecture \ref{BEH Conjecture} is still open even if we assume $c=5$ and that $M$ is cyclic. More precisely, it is still open whether or not $\beta_2 \geqslant {5 \choose 2} = 10$.

\item One could suppose further that $\beta_1 = 6$, so $M = R/I$ is an almost complete intersection. A result of Kunz \cite{Kunz} guarantees that $R/I$ is not Gorenstein, and thus $\beta_5 \geqslant 2$. Using linkage, Dugger \cite{Dugger} was able to show in this case that $\beta_2 \geqslant 9$.
\item In general, for cyclic modules, the rank of $\pi_1$ will be 1, and thus the Syzygy Theorem implies that 
$$\beta_2 = \rank \pi_2 + \rank \pi_3 \geqslant \rank \pi_2 + 3 = (\beta_1 - \rank \pi_1) + 3 = \beta_1 + 2,$$
so whenever $\beta_0 = 1$ and $\beta_1 \geqslant 8$ we will have the BEH bound for $\beta_2$.
\end{itemize}

\vspace{1em}

We close out this section with another general technique and an application.  Let $M$ be a graded module and $P$ be a prime ideal in its support. Since localization is exact, any minimal free resolution of $M$ over $R$ will remain exact upon localization at $P$. Hence, over the local ring $R_P$, the minimal free resolution of $M_P$ must be a direct summand of this resolution.  In other words, 
$$\beta_i^{R_P} (M_P) \leqslant \beta_i^R(M).$$

We now give two applications of this idea. The first shows that if we wanted to prove a stronger version of the BEH conjecture, we could restrict to finite length modules. 

\begin{conjecture}[Local BEH Conjecture]\label{Strong BEH}
Let $R$ be a local ring and $M$ a finitely generated $R$-module of codimension $c$. Then for all $i$,
$$\beta_i(M) \geqslant {c \choose i}.$$
\end{conjecture}

\begin{lemma}\label{Lemma:It Suffices to prove BEH for Finite Length}
To prove Conjecture \ref{Strong BEH}, it suffices to prove it for modules of finite length.
\end{lemma}

\begin{proof}
Let $M$ be an arbitrary module, not necessarily of finite length. Say that $M$ has codimension $c$, and note that there must be a minimal prime $P$ of $M$ of height $c$. Then $M_P$ is a finite length module over $R_P$, and 
$$\beta_i^R(M) \geqslant \beta_i^{R_P} (M_P)$$
by our localization argument. Since $M_P$ must then have codimension $c$, the result follows.
\end{proof}

We apply this idea to the case of monomial ideals and present a short proof that the BEH conjecture holds for monomial ideals.  As we will see in Section \ref{Section strong bounds}, there are in fact stronger bounds that hold in the monomial case.  

\begin{theorem}\label{theorem: the monomial ideal case easy proof with localization}
Let $I$ be a monomial ideal of height $c$ in a polynomial ring $R$.  Then the BEH conjecture holds and 
$\beta_i(R/I) \geqslant {c \choose i}.$
\end{theorem}

\begin{proof}
Our first step is to reduce to the case that $I$ is squarefree.  Indeed, if $I$ is a monomial ideal, then there is a squarefree monomial ideal (perhaps in a larger number of variables) called the \emph{polarization} of $I$ which has the same codimension and Betti numbers as $I$.

So consider the primary decomposition of a squarefree monomial ideal.  It consists entirely of minimal primes that are generated by subsets of the variables, and all must have height at least $c$. Choose any one you like and call it $P$. Note that $I_P = P_P$, since $P$ is minimal. Without loss of generality, we can assume $P = (x_1, \ldots, x_r)$ for some $r \geqslant c$. Then upon localizing $R/I$ at $P$, it is easy to see that 
$$(R/I)_P \cong R[x_1, \ldots, x_r]_{(x_1, \ldots, x_r)}/(x_1, \ldots, x_r),$$
whose Betti number are obtained from the Koszul complex on $x_1, \ldots, x_r$.
Thus 
$$\beta_i(R/I) \geqslant \beta_i(R_P/I_P) = {r \choose i} \geqslant {c \choose i}.$$
The reader will note that if we choose $r$ as large as possible, then $r$ would be the \emph{big height} of the squarefree monomial ideal $I$, that is, the largest height of an associated prime.
\end{proof}
\begin{remark}
Notice that it is \emph{not} clear that to prove the original BEH Conjecture (which was stated over a polynomial ring) one can simply study finite length modules.  Indeed, this localization argument might require one to work over localizations of polynomial rings, which despite being regular will not be polynomial rings.
\end{remark}

Finally, we include another important general result that comes up frequently.   As motivation we refer to Example \ref{Ex:3 lines in k^3} with $I = (xy,xz,yz)$.  Notice that the element $\ell = x - y - z$ is a regular element on $M = R/I$, for instance by looking at a primary decomposition.  If we work over $\overline{R} = R/(\ell)\cong k[y,z]$, then $\overline{M} \cong \overline{R}/(y^2,yz,z^2)$, which is a module of finite length. Standard arguments show that when we go modulo a regular element like this, the homological invariants (including the Betti numbers) do not change.  One application of this is the fact that the Betti numbers of Cohen-Macaulay modules are the same as those of finite length modules.  We make this sentence precise in the following:
\begin{proposition}\label{prop: Betti numbers of cohen macaulay are finite length}
Let $M$ be a Cohen-Macaulay module of codimension $c$ over the polynomial ring $R = k[x_1, \ldots, x_n]$ where $k$ is any field.   
There exist a field $k'$and a finite length module $M'$ over the polynomial ring $R' = k'[y_1, \ldots, y_{c}]$ such that the Betti numbers of $M$ and $M'$ coincide. Thus the following sets are equal: 
\begin{center}
$\{ \beta_i(M)  :  M \textrm{ Cohen-Macaulay of codimension } c \textrm{ over } k[x_1, \ldots, x_n] \textrm{ for some } k\}$

$\veq$

$\{ \beta_i(M)  :  M \textrm{ is finite length over } k[x_1,\ldots, x_c] \textrm{ for some } k \}.$
\end{center}
\end{proposition}

\begin{proof}
Let $M$ be a Cohen-Macaulay module of codimension $c$ over $k[x_1, \ldots, x_n]$. If $k$ is infinite, set $k' = k$.  If $k$ is finite, then we may enlarge the field, say to the algebraic closure $k ' = \overline{k}$, since flat base change will not affect the Betti numbers of $M$. Set $\overline{R} = k'[x_1, \ldots, x_n]$ and $\overline{M} = M \otimes_R \overline{R}$, where $\overline{M}$ is regarded as an $\overline{R}$-module.  Note that 
$\beta_i(M) = \beta_i(\overline{M})$. Now, since we are working over an infinite field, there is a sequence of linear forms $\ell_1, \ldots, \ell_{n-c}\in \overline{R}$ that is a maximal regular sequence on $\overline{M}$.  Now let $R' = \overline{R}/(\ell_1, \ldots, \ell_{n-c})$ and set
$$M' = \overline{M} \otimes_{\overline{R}} R'.$$
Then since we have gone modulo a regular sequence, $\beta_i(M') = \beta_i(M)$, and since the $\ell_i$ were linear forms, $R'$ is isomorphic to a polynomial ring $k'[y_1, \ldots, y_c]$. 
\end{proof}

\subsection{Other Results}

As we mentioned in the previous section, the BEH conjecture \ref{BEH Conjecture} remains open for modules of codimension $c \geqslant 5$ except in a small collection of cases. There are, however, some classes of modules for which the BEH Conjecture is known.

A deformation argument was used in \cite[Remark 4.14]{HU} to show that the conjecture holds for arbitrary $c$ when $M = R/I$ and $I$ is in the linkage class of a complete intersection. Additionally, in \cite{Erman} it was shown that if the regularity of $M$ is small relative to the degrees of the first syzygies of $M$, meaning the entries in a presentation matrix for $M$, then the BEH conjecture holds.  This will be discussed more carefully in Section \ref{Section strong bounds}.

The conjecture holds also when $M$ is multigraded, meaning that $M$ remains graded no matter what weights the generators $x_i$ are given. In fact, there are several proofs of this fact, for example \cite{Chara, CE, Santoni}, but perhaps the strongest version is the result due to Brun and R\"omer, \cite{BR} which shows that if $M$ is multigraded, then in fact $\beta_i(M) \geqslant {p \choose i}$, where $p$ is the projective dimension of $M$.  Since the projective dimension can exceed the codimension, this is a much stronger result.  Such a result cannot hold more generally --- after all, there are 3-generated ideals $I$ with projective dimension $1000$, by Theorem \ref{Bruns3Generated}, and in that case $\beta_1(R/I) = 3 < { 1000 \choose 1}$.  Nevertheless, it would be interesting to know if there are other classes where ${p \choose i}$ is a lower bound for the Betti numbers.  We know of at least one other case, when the resolution of $R/I$ is linear, which we present in Theorem \ref{Theorem HerzogKuhlLinear}.  We will discuss the multigraded case in more detail in Section \ref{Section strong bounds}, when we discuss stronger bounds on Betti numbers. 

Finally, the BEH conjecture \ref{BEH Conjecture} also holds for finite length modules of Loewy length 2 over any regular local ring $(R,\fm)$, meaning modules $M$ satisfying $\fm^2 M = 0$ \cite{Chang,Burman}.

\subsection{The Total Rank Conjecture}
If the Buchsbaum--Eisenbud--Horrocks Rank Conjecture is true, an immediate corollary would be the Total Rank Conjecture, which is obtained by adding the individual inequalities:

\begin{conjecture}[Total Rank Conjecture]\label{total rank conjecture}
If $M \neq 0$ is a finitely generated graded module over $R=k[x_1, \ldots, x_n]$ of codimension $c$, then 
$$\sum_{i = 0}^c \beta_i(M) \geqslant 2^c.$$
\end{conjecture}

This Conjecture was settled in 2018 by Walker \cite{W}, except in the case that $k$ has characteristic $2$. In fact, Walker's result also applies to finitely generated modules over an arbitrary local ring $R$ containing a field of odd characteristic. This result truly was a breakthrough in the field.

Even though the Total Rank Conjecture is settled (except in characteristic two), we cannot resist sharing some of the beautiful historical results in this story and compare them with the modern treatment. For example, the odd length case has a simple solution via elementary methods:

\begin{lemma}\label{Lemma:OddLength}
	Suppose that $M$ is a finitely generated $R$-module of (finite) odd length over $R=k[x_1, \ldots, x_n]$. Then 
	$$\sum_{i = 0}^n \beta_i(M) \geqslant 2^n.$$
\end{lemma}

\begin{proof}
The Hilbert series $h_M(t)$ of $M$ is a polynomial in $t$, say $h_M(t) = h_0 + h_1 t + \cdots + h_r t^r$. We can also write it as
	$$h_M^R(t) = \frac{\sum_{i,j} (-1)^i \beta_{i,j}(M) t^j}{(1-t)^n}.$$
	Plugging in $t=-1$, we obtain
	$$2^n h_M^R(-1) = \sum_{i,j} (-1)^{i+j} \beta_{i,j}(M),$$
	so
	$$2^n \left|h_0 - h_1 + \cdots + (-1)^r h_r \right| = \left| \sum_{i,j} (-1)^{i+j} \beta_{i,j}(M) \right| \leqslant \sum_i \beta_i(M).$$
	On the other hand, $h_0 + h_1 + \cdots + h_r$ is the rank of $M$, which we assumed to be odd. Therefore, $h_0 - h_1 + \cdots +(-1)^r h_r$ is also odd, and thus nonzero. In particular,\begin{equation*}
	2^n \leqslant \sum_{i=0}^n \beta_i(M).\qedhere\end{equation*}
\end{proof}

In other words, for modules of finite odd length, the Total Rank Conjecture holds simply due to constraints on its Hilbert function.  In 1993, Avramov and Buchweitz were able to obtain a generalization of this fact in \cite{ABO}. Their most general bound was that if $d\geqslant 5$ and $M$ is a module of finite length over $R$, then
$$\sum_{i=0}^d \beta_i(M) \geqslant \frac32 (d-1)^2 + 8.$$
In particular, this shows that when $d =5$ the lower bound of $32 = 2^5$ in the Total Rank Conjecture does hold. Their results were in fact much finer, depending on the prime factors of the length of $M$, $\ell(M)$.  For instance, they show that
\begin{itemize}
\item If $\ell(M)$ is odd, then $\sum \beta_i(M) \geqslant 2^d$, so they recover the above result.
\item If $\ell(M)$ is even but not divisible by 6, then 
$\sum \beta_i(M) \geqslant 3^{d/2} \geqslant 2^{0.79d}$.
\item If $\ell(M)$ is divisible by 6 but not by 30, then 
$\sum \beta_i(M) \geqslant 5^{d/4} \geqslant 2^{0.58d}$.
\item If $\ell(M)$ is divisible by 30 but not by 60, then 
$\sum \beta_i(M) \geqslant 2^{(d+1)/2}$.
\end{itemize}

If we move forward 25 years, the following is a summary of 
Walker's results:

\begin{theorem}[Walker, 2018 \cite{W,Walkerv1}]
Let $M$ be a finitely generated module of codimension $c$ over $k[x_1, \ldots, x_n]$. 
\begin{itemize}
\item If $\Char{k} \neq 2$, then $\sum \beta_i(M) \geqslant 2^c.$
\item If $\Char(k) = 2$, then $\sum \beta_i(M) \geqslant 2{(\sqrt3)}^{c-1} > 
2^{0.79c+0.208}.$
\end{itemize}
\end{theorem}

While the Total Rank Conjecture remains open in characteristic $2$, for that case Walker \cite[Theorem 5]{Walkerv1} did give the above bound of $2{(\sqrt3)}^{d-1}$, which improves the previous bounds by Avramov and Buchweitz \cite{ABO}.   We also remark that the Total Rank Conjecture is related to the Toral Rank Conjecture of Halperin \cite{Halperin}. For a survey on this and related results, see \cite{Munoz,Carlsson3,Carlsson2,Carlsson1}.

In the following table, we indicate the current status (as of the writing of this survey) of both the Total Rank Conjecture \ref{total rank conjecture} and the BEH Conjecture \ref{BEH Conjecture}.  The reader may want to refer to Table \ref{table: not a ci} at the end of the paper to see what the case for stronger bounds is.

\begin{savenotes}
\begin{table}[H]
\centering
{\renewcommand{\arraystretch}{2}
\begin{tabular}{| l | c | c |}
    \cline{2-3}
\multicolumn{1}{c|}{}
    & $c \leqslant 4$ & $c \geqslant 5$ \\
    \hline
    \multirow{3}{*}{$\beta_i \geqslant {c \choose i}$} & follows from the &  \\
    & Syzygy Theorem & \textcolor{purple}{Open}\\
    & (Evans--Griffith 1981) \cite{EG}& \\
    \cline{1-3}
    \multirow{2}{*}{$\sum_i \beta_i \geqslant 2^c$} & \multirow{2}{*}{\textcolor{OliveGreen}{follows from box above \checkmark}} & $c=5$ (Avramov--Buchweitz, 1993) \cite{ABO} \\
    && all $c$ char$(k) \neq 2$ (Walker, 2018) \cite{W} \\
    \hline
\end{tabular}
}
\caption{Status of the BEH and Total Rank Conjectures for a module of codimension $c$}
\label{table: BEH vs total rank}
    \end{table}
    \end{savenotes}

\newpage \section{Stronger bounds}\label{Section strong bounds}
We now turn to the question of whether there are larger bounds for Betti numbers and whether or not these bounds are achieved, starting with Walker's result \cite{W}.  

\begin{theorem}[Walker, 2018, Theorem 1 in \cite{W}]
Suppose that $\Char k \neq 2$, and let $M$ be a finitely generated graded $k[x_1, \ldots, x_n]$-module of codimension $c$.
Then $$\sum_{i=0}^c \beta_i(M) \geqslant 2^c$$
with equality if and only if $M$ is not a complete intersection. 
\end{theorem}

\begin{remark}
The situation where we have a module $M \cong R/I$ with $I$ an ideal generated by a regular sequence is very important, and we will want to distinguish it from any other kind of module; we will abuse notation\footnote{This is an abuse of notation since the expression ``complete intersection'' typically refers to a ring, not a module.} and say that $M$ is a {\bf complete intersection}. We will say that a module $M$ is not a complete intersection whenever $M$ is not isomorphic to any quotient of $R$ by a regular sequence; in particular, when we refer to modules $M$ that are not a complete intersection, we will include any non-cyclic module.
\end{remark}

Notice that this theorem says that the only time that the Betti numbers sum to $2^c$ is in the case of a complete intersection. Surprisingly, the next smallest value for the sum of the Betti numbers that we know of is $2^c + 2^{c-1}$, which is 50\% larger than the bound of $2^c$.  The next two examples show how to achieve this value. Notice that in one example this stems from the fact that $1 + 3 + 2 = 6$, whereas in the other it is because $1 + 5 + 5 + 1 = 12$.  

\begin{example}\label{Ex:132 and cones}
Let $I$ be the ideal $(x^2,xy,y^2)$ in $R = k[x,y]$.  Then $R/I$ is a finite length module of codimension $c =2$ with Betti numbers $\{1,3,2\}$.  Notice that these sum to $6$ which is $2^2 + 2^{2-1}$.  

By adding new variables (to $R$ and also to $I$) we can extend this example to any $c\geqslant 2$.  Indeed, set $R = k[x,y,z_1, \ldots, z_{c-2}]$, and let $I= (x^2, xy,y^2, z_1^2,z_2^2, \ldots, z^2_{c-2})$. Then the minimal free resolution of $R/I$ is obtained by tensoring the Koszul complex on $\{z_1, \ldots, z_{c-2}\}$ with the minimal free resolution of $R/(x^2,xy,y^2)$. Thus
\vspace{0.5em}
$$	\beta_i(R/I) =  {c-2 \choose i} + 3{ c-2\choose i-1} +2{c-2 \choose i-2}  $$
and we see that $\sum \beta_i(R/I) = 2^c+2^{c-1}.$
We chose to adjoin $z_i^2$ just so that our generators were all in the same degree, but one could choose these additional generators to be of any degree. Note that in all of these examples $I$ is {\bf monomial} and $R/I$ is {\bf of finite length.}
\end{example}

\begin{example}\label{Ex: the flength ideal with 1551} 
Consider the ideal 
$$G = (x^2,y^2,z^2, xy-yz,yz-xy)$$
in the ring $R = k[x,y,z]$.   The height of $G$ is 3 and the Betti numbers of $R/I$ are $\{1, 5, 5, 1\}$.  Note that $1 + 5 + 5 + 1 = 2^3 + 2^2$.  

Let $c \geqslant 3$.  Then as in Example \ref{Ex:132 and cones}, we can just add generators in new variables, say 
$$I = G + (z_1^2, \ldots, z_{c-3}^2)$$
and after tensoring with a Koszul complex we have that
$$\beta_i(R/I) = 
{c-3 \choose i-3} + 5{c-3 \choose i-2} + 5{c-3 \choose i-1} + {c-3 \choose i}.
$$
Therefore, 
$$\sum \beta_i(R/I) = 2^c+2^{c-1}.$$  
Note that all of the modules $R/I$ in this example are {\bf of finite length.}
\end{example}

\begin{example}\label{Ex: 1551 Monomial}
If one repeats Example \ref{Ex: the flength ideal with 1551} with $R = k[x,y,z,u,v]$ and $J = (xy,yz,zu,uv,vx)$ playing the role of $G$, then the numerics are exactly the same. $R/J$ has codimension $3$ and the Betti numbers are $\{1,5,5,1\}$.    The analogous examples obtained by adding new generators will all be monomial but {\bf not of finite colength.}  This distinction is important, because we will later see in Corollary \ref{HaraStrongBoundGamma} that there are bounds on the individual Betti numbers for monomial ideal of finite colength that do not hold for monomial ideals more generally, nor for general ideals of finite colength. 
\end{example}

The following result in \cite{CEM} shows that for modules that are not complete intersections, this behavior of Betti numbers adding up to ``50\% more than $2^c$'' does hold for $c\leqslant 4$:

\begin{theorem}[Charalambous--Evans--Miller, 1990, Theorem 3 in \cite{CEM}]\label{thm: CEM  c  leq 4}
Let $M$ be a finitely generated graded module of height $c$ over a polynomial ring. Suppose $M$ is not a complete intersection. If $c\leqslant 4$, then $\sum \beta_i(M) \geqslant 2^c + 2^{c-1}$. 
\end{theorem}
In fact, \cite{CEM} actually provides minimal Betti sequences for each codimension.  For example, in codimension $c = 4$ they show that $\{\beta_0, \ldots, \beta_4\}$ must be bigger (entry by entry) than at least one of the following: 
$$\begin{array}{cccc}\{1,5,9,7,2\}, & \{1,6,10,6,1\}, & \{2,6,8,6,2\}, & \{1,6,9,6,2\} \\ 
\{2,7,9,5,1\}, & & &  \{2,6,9,6,1\}\end{array}.$$ 
Note that the entries on the bottom row are the reverse of those directly above. The proof of this result uses techniques of linkage and relies on the classification \cite{KM} of the possible algebra structures on $\Tor_\bullet^R(R/I, k)$.  This result led the authors to ask the following question: 

\begin{question}[Charalambous--Evans--Miller, 1990]\label{Big Question 2c + 2c-1 bound}
If $M$ is finitely generated graded module over $k[x_1, \ldots, x_n]$ of codimension $c$ that is not a complete intersection, is 
$$\sum \beta_i(M) \geqslant (1.5) 2^c = 2^c + 2^{c-1}?$$
\end{question}

We will now discuss several instances where we have an affirmative answer to this question.  We remark, however, that the techniques --- and indeed the underlying reasons --- in each instance are completely different!  Here are some natural follow-up questions.

\begin{question}
What other modules $M$ of codimension $c$ satisfy $\sum \beta_i(M) = (1.5) 2^c$?
\end{question}

\begin{question}
	What are the smallest Betti sequences in a given codimension $c$, when we range over all finitely generated modules of codimension $c$ over a polynomial ring on any number of variables?
\end{question}

\subsection{The Multigraded Case} Let $R = k[x_1, \ldots, x_n]$ and let $M$ be a finitely generated graded-module over $R$.  We say that $M$ is {\bf multigraded} if it remains graded with respect to any grading of the variables.   For example, when $I$ is a monomial ideal, $I$ and $R/I$ are multigraded (each).

\begin{example}\label{Ex:2442 Example}
Let $R = k[x,y,z]$. Consider
$$M = \coker \begin{pmatrix} y & 0  & z \\ -x & z & 0 \\ 0 & -y & -x
\end{pmatrix}.$$
The module $M$ is generated by $3$ elements; for example, since $M$ is a quotient of $R^3$, we can take the images of the canonical basis elements $e_1, e_2, e_3$ for generators of $M$. Then $M$ has relations
$$ye_1 = xe_2, \ \ ze_2 = ye_3, \ \ ze_1 = xe_3.$$
Suppose we are given any weights on the variables. Then the module $M$ will be graded as well by setting $\deg e_1 := \deg(x), \deg e_2 := \deg(y), \deg e_3 := \deg(z)$.    

In contrast, consider
$$N = \coker \begin{pmatrix} x & y&  z & 0 \\ 0 & x & y & z 
\end{pmatrix}, \hspace{4em}
\begin{array}{r|cccc} 
\beta(N) &0 &1&2&3\\ \hline
0 & 2 & 4 & - & - \\
 1 & - & - & 4 & 2 \\
\end{array}.$$
Then $N$ is generated by $e_1, e_2$, and the relations   
$$ye_1 + xe_2= 0, \quad ze_1 +ye_2 = 0$$ 
imply that 
$$\deg(y) + \deg(e_1) = \deg(x) + \deg(e_2) $$
$$\deg(z) + \deg(e_1) = \deg(y) + \deg(e_2) $$
which has no solution for example when $\deg(y) = \deg(x) = 0$ and $\deg (z) = 1$. In other words, $N$ is not multigraded.  Notice that $N$ is finite length as an $R$-module, and thus has codimension $3$.
\end{example}

The following theorem gives strong bounds on the individual Betti numbers of modules that are multigraded and of finite-length.   For instance they imply that a module of codimension $c$ must have Betti numbers that either exceed $\{1, 4, 5, 2\}$ or $\{2, 5, 4, 1\}$.    Noting that the Betti numbers of $N$ in the previous example violate both of these bounds provides yet another reason why it is not multigraded.

\begin{theorem}[Charalambous--Evans, 1991 \cite{CE}]\label{HaraStrongBoundGamma} Let $M$ be a multigraded module {\bf of finite length} and let $\gamma_i(M)$ denote the rank of the $i$th syzygy module of $M$.  Then for all $i$ 
$$\gamma_i(M) \geqslant {n - 1 \choose i-1}\mbox{ and therefore } \beta_i(M) \geqslant {n \choose i}.$$
Further if $M$ is not a complete intersection, then at least one of the following holds:
\begin{enumerate}[(a)]
	\item for all $i$, $\gamma_i(M) \geqslant {n-1 \choose i-1} + {n-2 \choose i-2}$, and therefore $\beta_i(M) \geqslant {n \choose i} + {n-1 \choose i-1}$;
	\item  for all $i$, $\gamma_i(M) \geqslant {n-1 \choose i-1} + {n-2 \choose i-1}$, and therefore $\beta_i(M) \geqslant {n \choose i} + {n-1 \choose i}$.
\end{enumerate}
\end{theorem}

\begin{remark}\label{remark that shows how we violate stronger bounds}
We want to emphasize that without the assumption that $M$ is multigraded \emph{and} of finite length, Theorem \ref{HaraStrongBoundGamma} is false if $c \geqslant 3$.  Indeed, Examples \ref{Ex: the flength ideal with 1551} (respectively \ref{Ex: 1551 Monomial}) give families of modules $R/I$ that are finite length (respectively multigraded) but with Betti numbers that violate the bounds in Theorem \ref{HaraStrongBoundGamma}.  This is essentially due to the fact that $R/I$ is Gorenstein in both cases.  Indeed, since Theorem \ref{HaraStrongBoundGamma} implies that either $\beta_0(M) \geqslant 2$ or $\beta_c(M) \geqslant 2$, any Gorenstein algebra $R/I$ that is not a complete intersection will violate the bounds in Corollary \ref{CorHaraSumBettiMultigraded}.  In fact, we can use this to deduce the following classical fact: if $I$ is a monomial ideal in a polynomial ring $R$ such that $R/I$ is of finite length and Gorenstein, then $R/I$ is a complete intersection.  
\end{remark}

Summing the inequalities for the Betti numbers in Theorem \ref{HaraStrongBoundGamma} yields the following result, which is a special case of Question \ref{Big Question 2c + 2c-1 bound}.  

\begin{corollary}\label{CorHaraSumBettiMultigraded}
If $M$ is a \textbf{multigraded} module {\bf of finite length} then 
$$\sum \beta_i(M) \geqslant 2^n.$$
Further if $M$ is not a complete intersection, then
$$\sum_{i=0}^n \beta_i(M) \geqslant 2^n + 2^{n-1}.$$
\end{corollary}

We remark that in this case $n= \codim M$.   

\

Notice that one of the examples in Remark \ref{remark that shows how we violate stronger bounds} has Betti numbers $\{1,5,5,1\}$, and although this violates the bounds in Theorem \ref{HaraStrongBoundGamma}, they nonetheless add up to $2^3 + 2^2$.  Recently, the first author and Seiner were able to show that one can remove the finite length assumption, provided one works with multigraded cyclic modules: 

\begin{theorem}[Boocher--Seiner, 2018 \cite{BS}] 
Let $I\subseteq R =k[x_1, \ldots,x_n]$ be a monomial ideal of any codimension $c \geqslant 2$.  If $R/I$ is not a complete intersection, then
$$\sum_{i=0}^c \beta_i(R/I) \geqslant 2^c + 2^{c-1}.$$
\end{theorem}

Unlike the proofs in the finite length case, this theorem does not apparently follow from a bound on the individual Betti numbers.  Indeed, the argument follows via a degeneration argument that reduces everything to either a Betti sequence $\{1,3,2\}$ with $c=2$ or a Betti sequence $\{1,5,5,1\}$ with $c=3$.  Perhaps it is a coincidence that these Betti numbers sum to $(1.5)2^c$.

\begin{question}
	Examples \ref{Ex: the flength ideal with 1551} and \ref{Ex: 1551 Monomial} are both examples of Gorenstein algebras where the sum of the Betti numbers is equal to $2^c + 2^{c-1}$.  What other Gorenstein algebras $R/I$ of codimension $c$ have this sum?
\end{question}

\begin{question} In Examples \ref{Ex:132 and cones} and \ref{Ex: the flength ideal with 1551}, we saw two distinct families of Betti numbers whose Betti numbers sum to $2^c + 2^{c-1}$.  Are there other examples of Betti numbers that achieve this sum?
\end{question}

\begin{question}\label{QuestionForMultigradedOpen} If $M$ is a {\bf multigraded} $k[x_1, \ldots, x_n]$-module of codimension $c <n$ that is not a complete intersection, then does
$$\sum \beta_i(M) \geqslant 2^c + 2^{c-1}$$
hold?
\end{question}
As a partial answer to this, we have the following:

\begin{theorem}[Brun--Romer, 2004 \cite{BR}]\label{BrunRomer}
If $M$ is multigraded $k[x_1, \ldots, x_n]$-module of projective dimension $p$, then $\beta_i(M) \geqslant {p \choose i}.$ 
\end{theorem}

Since $p\geqslant c$ with equality only in the case that $M$ is Cohen-Macaulay, we see that Question \ref{QuestionForMultigradedOpen} can be reduced to the Cohen-Macaulay case. 

Finally, we cannot resist including the following beautiful result of Charalambous and Evans, which gives a sharp strong bound for monomial ideals of finite colength: 

\begin{theorem}[Charalambous--Evans, 1991 \cite{CE}]
Let $R = k[x_1, \ldots, x_n]$ and $M = R/I$, where the ideal $I$ is minimally generated by $n$ pure powers of the variables and one additional generator $m = x_1^{a_1}\cdots x_n^{a_n}$.  Suppose that $\ell$ is the number of nonzero $a_i$'s.  Then for all $i$, we have
$$\beta_i(M) = {n \choose i} + {n-1 \choose i -1 } + \cdots + {n - (\ell -1 ) \choose i - (\ell - 1)}.$$
\end{theorem}
For instance, this says that the Betti numbers of the ideal $I = (x^2, y^2, z^2, w^2, xywz)$ must sum to at least $2^4 + 2^3 + 2^2 + 2 = 30$.  Indeed, the Betti numbers are $\{1,5,10, 10, 4\}$. 

\begin{question}
Can this theorem be extended outside of the case of finite colength monomial ideals?  Is there a version for general monomial ideals?  Is there a version for multigraded modules?  For general ideals? 
\end{question}

\subsection{Low Regularity Case}
We finish this survey with some of the most recent results on larger lower bounds for Betti numbers.  So far we have not paid much attention to the degrees of the syzygies.  After all, our bounds are in terms of the Betti numbers $\beta_i$, which count the number of generators, but not their degrees, of the $i$th syzygy module.  But since we are working with graded modules, we will now actually look at $\beta_{ij}$. 

In terms of degrees, the simplest resolutions are those whose matrices all have linear entries.  Such resolutions are called linear.  

\begin{theorem}[Herzog--K\"uhl, 1984 \cite{HK}]\label{Theorem HerzogKuhlLinear} 
If $M$ is a graded $R$-module of projective dimension $p$ with a linear resolution, then $\beta_i(M) \geqslant {p \choose i}$. 
\end{theorem}

\begin{remark}
Notice that this is the same bound given by Brun and R\"omer for multigraded modules in Theorem \ref{BrunRomer}. In the same paper, Herzog and K\"uhl show that apart from this bound, linear resolutions can behave quite wildly.\footnote{This is the term used by Herzog and K\"uhl.} Indeed, they show how to produce squarefree monomial ideals with a linear resolution such that the Betti numbers form a non-unimodal sequence with arbitrarily many extrema.
\end{remark}

Linear resolutions have the property that each matrix has entries all of which are linear.  This is a particular case of what is called a pure resolution.  We say that a module $M$ is pure if it is Cohen-Macaulay and each map has entries all of the same degree.  Equivalently, in the free resolution $ F_\bullet\rightarrow M$, each $F_i$ is generated in a single degree $d_i$.  This sequence of numbers $\{d_0, \ldots, d_c\}$ is called the degree sequence of $M$.   

\begin{example}\label{Ex:Two examples of pure modules, 2442 and 1551}
The module $M$ given in Example \ref{Ex:2442 Example} with Betti table
$$\begin{array}{r|cccc} 
\beta(M) &0 &1&2&3\\ \hline
0 & 2 & 4 & - & - \\
 1 & - & - & 4 & 2 \\
\end{array}$$
is pure with degree sequence $0,1,3,4$.  The module $R/G$ in Example \ref{Ex: the flength ideal with 1551}  is an example of a pure module with degree sequences $\{0,2,3,5\}$ and Betti table
$$\begin{array}{r|cccc} 
\beta(M) &0 &1&2&3\\ \hline
0 & 1 & - & - & - \\
 1 & - & 5 & 5 & - \\
  2 & - & - & - & 1 \\
\end{array}.$$
\end{example}

In \cite{HK}, Herzog and K\"uhl showed that if $M$ is a pure module with degree sequence $\{d_0, \ldots, d_c\}$, then for all $i\geqslant 1$ we have
$$\beta_i(M) = \beta_0(M) \prod_{\stackrel{1\leqslant j \leqslant c}{j \neq i}}\frac{|d_j-d_0|}{|d_i-d_j|}.$$

Quite surprisingly, given any degree sequence $d_0< d_1 < \cdots < d_p$, there exists a Cohen-Macaulay module $M$ whose resolution is pure with this degree sequence. This was proven in \cite{ES,EFW} as part of the resolution of the Boij-S\"oderberg conjectures.

\begin{question}\label{question for pure} If $M$ is a pure module of codimension $c$, is $\beta_i(M) \geqslant { c \choose i}$?
\end{question}

Given the Herzog-K\"uhl equations, one might expect that this question is numerical in nature, and in a sense it is. However, the following example shows a major obstacle: 

\begin{example}\label{Ex:Pure modules with small Betti numbers two examples 32 and 44}
Let $M$ be a pure module with degree sequence $\{0,2,3,7,8,10\}$.  Such a module has codimension $5$ and its Betti table is
$$\begin{array}{r|cccccc} 
\beta(M) &0 &1&2&3& 4 & 5\\ \hline
0 & \beta_0(M) & - & - & - & - & - \\
 1 & - & 7\beta_0(M) & 8\beta_0(M) & - & - & - \\
  2 & - & - & - & -& - & -  \\
  3& - & - & - & - & - & -  \\
  4& - & - & - & 8\beta_0(M) & 7\beta_0(M) & -  \\
  5& - & - & - & - & - & \beta_0(M)  \\
\end{array}.$$
Notice that if $\beta_0(M) = 1$, then this would give an example of a module with $\beta_2(M) < { 5\choose 2}$.  So part of answering Question \ref{question for pure} involves showing that $\beta_0(M) \geqslant 2$.  One way to prove this is to apply a big hammer --- the Total Rank Conjecture, now Walker's Theorem \cite{W}. Using Walker's Theorem, we notice that if $\beta_0(M) = 1$, then the sum of the Betti numbers would be equal to $2^c$, but evidently $M$ is not a complete intersection, which contradicts Walker's result.  Alternatively, one could note that from the Betti table, the rank of $\Omega_3(M)$ would be $2\beta_0(M)$, which would violate the Syzygy Theorem \ref{The SyzygyTheorem} when $\beta_0(M) = 1$. 

Extending this sort of argument to general degree sequences will present many challenges. In fact, we need only to turn to the degree sequence $\{0,1,2,3,5,7,8,9,10\}$ to see to limits of this argument.  A module $M$ possessing a pure resolution with this degree sequence would necessarily be of codimension $8$ and would have Betti table 
$$\begin{array}{r|ccccccccc} 
\beta(M) &0 &1&2&3& 4 & 5& 6 & 7 & 8\\ \hline
 0  &  4N &  25N & 60N & 60N & -  & - & - & - & - \\
 1   & -& -& -& -& \textcolor{red}{42N}& -& -& - & - \\
 2 & -& -& -& -& -& 60N& 60N& 25N& 4N \\
\end{array}$$
for some positive integer $N$.  Boij-S\"oderberg Theory guarantees that such a module exists, but $N$ may be large.   Note that if $N = 1$ then  $\beta_4 < {8 \choose 4}$, and the sum of the Betti numbers would be $340 < 2^8 + 2^7$ which would violate both the BEH Conjecture \ref{BEH Conjecture} and provide a negative answer to Question \ref{Big Question 2c + 2c-1 bound}.  Notice that the Betti sequence is non-unimodal, regardless of $N$. 

\end{example}

The numerical behavior resulting from the Herzog-K\"uhl equations is nontrivial to analyze, but is slightly manageable in the case where the last degree $d_c$ is small relative to $d_1$.  Note that $d_1$ and $d_c$ are esssentially degrees of the first syzygies of $M$ and the Castelnuovo-Mumford Regularity.  This insight was first noticed by Erman in \cite{Erman}.  Coupling this observation with the full force of the newly proven Boij-S\"oderberg Theory allowed him to prove the BEH Conjecture for those graded modules whose regularity is low relative to the degrees of the first syzygies. 

\begin{theorem}[Erman, 2010 \cite{Erman}]
 Let $M$ be a graded $R$-module of codimension $c\geqslant 3$ generated in
  degree $0$ and let $a\geqslant 2$ be the minimal degree of a first syzygy
  of $M$.  If $\reg(M) \leqslant 2a-2$, then
$$\beta_i(M) \geqslant \beta_0(M){ c \choose i}.$$ 
In particular the sum of the Betti numbers is at least $\beta_0(M) 2^c$.
\end{theorem}

To put the regularity bound into perspective, if $M$ is $R/I$ for some ideal $I$ generated by quadrics, then the above theorem would apply to any $M$ with regularity at most $2$, which means the Betti table has at most 2 rows.  The regularity condition is relaxed enough to include, for example, the coordinate rings of smooth curves embedded by linear systems of high degree, those of toric surfaces, as well as any finite length module whose socle degree is relatively low.  In Example \ref{Ex:Pure modules with small Betti numbers two examples 32 and 44} the two Betti tables do not obey the low regularity bound.  In the first, $a = 2$ and $\reg(M) = 5$;  in the second, $a = 1$ and $\reg(M) = 2$.  

Erman's proof uses general Boij-S\"oderberg techniques to reduce studying the Betti tables of arbitrary modules to the study of pure modules and then use a degeneration argument to supply the required numerical bound. These techniques were pushed even further in \cite{BoocherW}, where it is shown that in fact the sum of the Betti numbers is 50\% larger:

\begin{theorem}[Boocher--Wigglesworth, 2020 \cite{BoocherW}]\label{BoocherW}
  Let $M$ be a graded $R$-module of codimension $c \geqslant 3$ generated in
  degree $0$ and let $a \geqslant 2$ be the minimal degree of a first syzygy
  of $M$.  If $\reg(M) \leqslant 2a-2$, then
$$\beta(M) \geqslant \beta_0(M)(2^c+2^{c-1}).$$
If moreover $c \geqslant 9$, then
$$\beta_i(M)\geqslant 2 \, \beta_0(M){c \choose i}$$
for the first half of the Betti numbers, meaning for $1\leqslant i \leqslant \lceil c/2\rceil$. 
\end{theorem}

Essentially, this says that if the regularity is ``low'', then for $c\geqslant 9$, the \emph{first half} of the Betti numbers are at least \emph{double} the conjectured Buchsbaum--Eisenbud--Horrcks bounds. Then on average the Betti numbers, will be at least $1.5$ times the BEH bounds, and thus the sum of all the Betti numbers needs to be at least $1.5(2^c)$.  The authors deal with the cases $c\leqslant 8$ separately.   Again it seems almost miraculous that the bound of $2^c + 2^{c-1}$ pops up --- in this case aided by the fact that the first half of the Betti numbers are twice as large as excepted.

\begin{remark}\label{rmk: c+1 vs 2c}
Notice that if $R = k[x_1, \ldots, x_c]$ with $c\geqslant 2$, then any ideal $I$ generated by $c+1$ generic quadrics will be an ideal of height $c$, and $\beta_1(R/I) = c + 1 < 2{c \choose 1}$. So without some other condition, for example on the regularity, there is no hope of finding a stronger bound for the first Betti number.
\end{remark}

As a corollary of Theorem \ref{BoocherW}, for ideals generated by quadrics with $c \geqslant 9$ we have
$$\reg(R/I) < 3, \mbox{ and $R/I$ is not a CI } \implies \beta_1(R/I) \geqslant 2c, \, \beta_2(R/I) \geqslant 2{c \choose 2}, \, \ldots $$ 
In other words, low regularity forces this rather specific bound for the number of generators.

\vspace{1em}

We end with a table summarizing the results concerning these stronger bounds (each).  We remind the reader that these entries all concern modules that are not complete intersections.

\noindent
   \begin{table}[H]
\centering
{\renewcommand{\arraystretch}{2}
\begin{tabular}{| c | c | c | c|c | c |}
    \cline{2-6}
\multicolumn{1}{c|}{}
    & $c \leqslant 4$ & $c \geqslant 5$ & \multicolumn{2}{c|}{multigraded} & low regularity \\
    \hline
    \multirow{4}{*}{$\begin{array}{c} \beta_i \geqslant {c \choose i} + {c-1 \choose i-1} \textrm{ for all } i \\ \textrm{ or } \\ \beta_i \geqslant {c \choose i} + {c-1 \choose i} \textrm{ for all } i
\end{array}$} & \multirow{4}{*}{$\begin{array}{c} \textrm{\textcolor{purple}{False}} \\ \textrm{\ref{remark that shows how we violate stronger bounds}} \end{array}$} & \multirow{4}{*}{$\begin{array}{c} \textrm{\textcolor{purple}{False}} \\ \textrm{\ref{remark that shows how we violate stronger bounds}}\end{array}$} & \multirow{7}{*}{$\begin{array}{c} c=n \\ \textrm{(CE, 1991) } \\ \textrm{\cite{CE}}\end{array}$} & \multirow{4}{*}{$\begin{array}{c} c < n \\ \textrm{\color{purple}{False}} \\ \textrm{\ref{remark that shows how we violate stronger bounds}} 
\end{array}$} & \multirow{4}{*}{$\begin{array}{c} \textrm{\color{purple}{False}} \\ \textrm{Rk \ref{remark that shows how we violate stronger bounds}} \end{array}$} \\
     &&& & &\\
     &&& &  &\\
     &&&&& \\
    \cline{1-3} \cline{5-6}
    \multirow{4}{*}{$\sum_i \beta_i \geqslant$ {\color{purple}$(1.5)$} $2^c$} & \multirow{3}{*}{$\begin{array}{c} \textrm{(CEM,} \\ \textrm{ 1990)} \\ \textrm{\cite{CEM}}\end{array}$} & \multirow{3}{*}{\color{purple}{Open}} & & \multirow{2}{*}{$\begin{array}{c} M \cong R/I \\ \textrm{(BS, '18 \cite{BS}) } \end{array}$} &  \multirow{7}{*}{$\xymatrix{\textrm{True} \\ \\ \textrm{True if } c \geqslant 9 \ar@{=>}[uu]_-{5 \leqslant c \leqslant 8}^-{\textrm{check}} \\ \textrm{(BW, '20 \cite{BoocherW})}}$} \\
    & & & & & \\
    \cline{5-5}
    &&&&\multirow{2}{*}{$\begin{array}{c} \textrm{general } M \\ \textrm{\color{purple}{Open}} \end{array}$} & \\
     &&&&& \\
    \cline{1-5}
    \multirow{3}{*}{$\beta_i \geqslant 2 {c \choose i}$ for $i<\frac{c}{2}$} & \multirow{3}{*}{$\begin{array}{c} \textrm{\color{purple}{False}} \\ \textrm{\ref{rmk: c+1 vs 2c} } \end{array}$}
 & \multirow{3}{*}{$\begin{array}{c} \textrm{\color{purple}{False}} \\ \textrm{\ref{rmk: c+1 vs 2c}} \end{array}$} & \multirow{3}{*}{$\begin{array}{c} \textrm{\color{purple}{False}} \\ \textrm{\ref{rmk: c+1 vs 2c}} \end{array}$} & \multirow{3}{*}{$\begin{array}{c} \textrm{\color{purple}{False}} \\ \textrm{\ref{rmk: c+1 vs 2c}} \end{array}$} & \\
     &&&&&\\
     &&&&&\\
    \hline
\end{tabular}
}
\caption{$M$ is a module of codimension $c$ that is not a complete intersection}
\label{table: not a ci}
\end{table}

\section*{Acknowledgements}

We thank Daniel Erman and Josh Pollitz for their detailed comments on a previous version of this survey. We also thank Mark Walker for helpful correspondence. The second author was partially supported by NSF grant DMS-2001445, now DMS-2140355.

\bibliographystyle{plain} 
\bibliography{References}

\begin{thebibliography}{10}

\bibitem{StillmansConjecture}
Tigran Ananyan and Melvin Hochster.
\newblock Small subalgebras of polynomial rings and {S}tillman's {C}onjecture.
\newblock {\em J. Amer. Math. Soc.}, 33(1):291--309, 2020.

\bibitem{Ardila}
Federico Ardila and Adam Boocher.
\newblock The closure of a linear space in a product of lines.
\newblock {\em J. Algebraic Combin.}, 43(1):199--235, 2016.

\bibitem{AuslanderBuchsbaum}
Maurice Auslander and David~A. Buchsbaum.
\newblock Homological dimension in local rings.
\newblock {\em Trans. Amer. Math. Soc.}, 85:390--405, 1957.

\bibitem{LuchoObstructions}
Luchezar~L. Avramov.
\newblock Obstructions to the existence of multiplicative structures on minimal
  free resolutions.
\newblock {\em Amer. J. Math.}, 103(1):1--31, 1981.

\bibitem{InfiniteFreeResolutions}
Luchezar~L Avramov.
\newblock Infinite free resolutions.
\newblock In {\em Six lectures on commutative algebra}, pages 1--118. Springer,
  1998.

\bibitem{ABO}
Luchezar~L. Avramov and Ragnar-Olaf Buchweitz.
\newblock Lower bounds for {B}etti numbers.
\newblock {\em Compositio Math.}, 86(2):147--158, 1993.

\bibitem{AvBuIy}
Luchezar~L. Avramov, Ragnar-Olaf Buchweitz, and Srikanth~B. Iyengar.
\newblock Class and rank of differential modules.
\newblock {\em Invent. Math.}, 169(1):1--35, 2007.

\bibitem{AKM}
Luchezar~L. Avramov, Andrew~R. Kustin, and Matthew Miller.
\newblock Poincar\'{e} series of modules over local rings of small embedding
  codepth or small linking number.
\newblock {\em J. Algebra}, 118(1):162--204, 1988.

\bibitem{Bigatti}
Anna~Maria Bigatti.
\newblock Upper bounds for the {B}etti numbers of a given {H}ilbert function.
\newblock {\em Comm. Algebra}, 21(7):2317--2334, 1993.

\bibitem{B}
Adam {Boocher}.
\newblock {Free resolutions and sparse determinantal ideals.}
\newblock {\em {Math. Res. Lett.}}, 19(4):805--821, 2012.

\bibitem{BS}
Adam Boocher and James Seiner.
\newblock Lower bounds for betti numbers of monomial ideals.
\newblock {\em Journal of Algebra}, 508:445--460, 2018.

\bibitem{BoocherW}
Adam Boocher and Derrick Wigglesworth.
\newblock Large lower bounds for the betti numbers of graded modules with low
  regularity.
\newblock {\em Collectanea Mathematica}, 72(2):393--410, 2021.

\bibitem{BrownErman}
Michael~K. Brown and Daniel Erman.
\newblock Minimal free resolutions of differential modules, 2021.

\bibitem{BR}
Morten Brun and Tim R{\"o}mer.
\newblock Betti numbers of $\mathbb{Z}^n$-graded modules.
\newblock {\em Communications in Algebra}, 32(12):4589--4599, 2004.

\bibitem{Bruns}
Winfried Bruns.
\newblock ``{J}ede'' endliche freie {A}ufl\"{o}sung ist freie {A}ufl\"{o}sung
  eines von drei {E}lementen erzeugten {I}deals.
\newblock {\em J. Algebra}, 39(2):429--439, 1976.

\bibitem{BrunsHerzog}
Winfried Bruns and J{\"u}rgen Herzog.
\newblock {\em {Cohen-{M}acaulay rings}}, volume~39 of {\em {Cambridge Studies
  in Advanced Mathematics}}.
\newblock Cambridge University Press, Cambridge, 1993.

\bibitem{BERemarks}
D~Buchsbaum and David Eisenbud.
\newblock Remarks on ideals and resolutions.
\newblock In {\em Symp. Math}, volume~11, pages 193--204, 1973.

\bibitem{BE}
David~A. Buchsbaum and David Eisenbud.
\newblock Algebra structures for finite free resolutions, and some structure
  theorems for ideals of codimension {$3$}.
\newblock {\em Amer. J. Math.}, 99(3):447--485, 1977.

\bibitem{BuchsbaumRim1963}
David~A. Buchsbaum and Dock~S. Rim.
\newblock A generalized koszul complex.
\newblock {\em Bull. Amer. Math. Soc.}, 69(3):382--385, 05 1963.

\bibitem{Burch}
Lindsay Burch.
\newblock A note on the homology of ideals generated by three elements in local
  rings.
\newblock {\em Proc. Cambridge Philos. Soc.}, 64:949--952, 1968.

\bibitem{Burman}
Jennifer Burman.
\newblock Chang's theorem on betti numbers of exponent-two modules over regular
  local rings.
\newblock {\em Communications in Algebra}, 39(2):718--729, 2011.

\bibitem{Carlsson3}
Gunnar Carlsson.
\newblock Free $(\mathbb{Z}/2)^3$-actions on finite complexes.
\newblock {\em Algebraic topology and algebraic K-theory (Princeton, NJ,
  1983)}, 113:332--344, 1983.

\bibitem{Carlsson1}
Gunnar Carlsson.
\newblock On the homology of finite free $(\mathbb{Z}/2)^n$-complexes.
\newblock {\em Inventiones mathematicae}, 74(1):139--147, 1983.

\bibitem{Carlsson2}
Gunnar Carlsson.
\newblock Free $(\mathbb{Z}/2)^k$-actions and a problem in commutative algebra.
\newblock In {\em Transformation Groups Pozna{\'n} 1985}, pages 79--83.
  Springer, 1986.

\bibitem{CavigliaLiang}
Giulio Caviglia and Yihui Liang.
\newblock Explicit stillman bounds for all degrees.
\newblock {\em arXiv preprint arXiv:2009.02826}, 2020.

\bibitem{LargestBettiPolynomial}
Giulio Caviglia and Satoshi Murai.
\newblock Sharp upper bounds for the {B}etti numbers of a given {H}ilbert
  polynomial.
\newblock {\em Algebra Number Theory}, 7(5):1019--1064, 2013.

\bibitem{Chang}
Shou-Te Chang.
\newblock Betti numbers of modules of exponent two over regular local rings.
\newblock {\em Journal of Algebra}, 193(2):640--659, 1997.

\bibitem{SundanceCollection}
H~Charalambous and G~Evans.
\newblock Problems on betti numbers of finite length modules.
\newblock {\em Free resolutions in commutative algebra and algebraic geometry
  (Sundance, UT, 1990)}, 2:25--33, 1992.

\bibitem{Chara}
Hara Charalambous.
\newblock Betti numbers of multigraded modules.
\newblock {\em Journal of Algebra}, 137(2):491--500, 1991.

\bibitem{CE}
Hara Charalambous and E.~Graham Evans.
\newblock A deformation theory approach to {B}etti numbers of finite length
  modules.
\newblock {\em J. Algebra}, 143(1):246--251, 1991.

\bibitem{CEM}
Hara Charalambous, E.~Graham Evans, and Matthew Miller.
\newblock Betti numbers for modules of finite length.
\newblock {\em Proc. Amer. Math. Soc.}, 109(1):63--70, 1990.

\bibitem{CDNG2}
A.~Conca, E.~De~Negri, and E.~Gorla.
\newblock Universal {G}r\"{o}bner bases and {C}artwright-{S}turmfels ideals.
\newblock {\em Int. Math. Res. Not. IMRN}, (7):1979--1991, 2020.

\bibitem{CDNG}
Aldo Conca, Emanuela De~Negri, and Elisa Gorla.
\newblock Universal {G}r\"{o}bner bases for maximal minors.
\newblock {\em Int. Math. Res. Not. IMRN}, (11):3245--3262, 2015.

\bibitem{ConcaVarbaro}
Aldo Conca and Matteo Varbaro.
\newblock Square-free {G}r\"{o}bner degenerations.
\newblock {\em Invent. Math.}, 221(3):713--730, 2020.

\bibitem{BoocherDevries}
Justin~W DeVries.
\newblock On the rank of multi-graded differential modules.
\newblock {\em arXiv preprint arXiv:1011.2167}, 2010.

\bibitem{Dugger}
Daniel Dugger.
\newblock Betti numbers of almost complete intersections.
\newblock {\em Illinois J. Math.}, 44(3):531--541, 2000.

\bibitem{Eisenbud}
David Eisenbud.
\newblock {\em Commutative algebra}, volume 150 of {\em Graduate Texts in
  Mathematics}.
\newblock Springer-Verlag, New York, 1995.
\newblock With a view toward algebraic geometry.

\bibitem{EFW}
David Eisenbud, Gunnar Fl{\o}ystad, and Jerzy Weyman.
\newblock The existence of equivariant pure free resolutions (existence de
  r{\'e}solutions pures et libres equivariantes).
\newblock In {\em Annales de l'institut Fourier}, volume~61, pages 905--926,
  2011.

\bibitem{ES}
David Eisenbud and Frank-Olaf Schreyer.
\newblock Betti numbers of graded modules and cohomology of vector bundles.
\newblock {\em Journal of the American Mathematical Society}, 22(3):859--888,
  2009.

\bibitem{Erman}
Daniel Erman.
\newblock A special case of the {B}uchsbaum-{E}isenbud-{H}orrocks rank
  conjecture.
\newblock {\em Math. Res. Lett.}, 17(6):1079--1089, 2010.

\bibitem{BigPolynomialRings}
Daniel Erman, Steven~V. Sam, and Andrew Snowden.
\newblock Big polynomial rings and {S}tillman's conjecture.
\newblock {\em Invent. Math.}, 218(2):413--439, 2019.

\bibitem{StillmansConjectureBulletin}
Daniel Erman, Steven~V. Sam, and Andrew Snowden.
\newblock Cubics in 10 variables vs. cubics in 1000 variables: uniformity
  phenomena for bounded degree polynomials.
\newblock {\em Bull. Amer. Math. Soc. (N.S.)}, 56(1):87--114, 2019.

\bibitem{EG}
E.~Graham Evans and Phillip Griffith.
\newblock The syzygy problem.
\newblock {\em Ann. of Math. (2)}, 114(2):323--333, 1981.

\bibitem{SyzygiesBook}
E~Graham Evans and Phillip Griffith.
\newblock {\em Syzygies}, volume 106.
\newblock Cambridge University Press, 1985.

\bibitem{Halperin}
Stephen Halperin.
\newblock Le complexe de koszul en algebre et topologie.
\newblock In {\em Annales de l'institut Fourier}, volume~37, pages 77--97,
  1987.

\bibitem{HartshorneProblems}
Robin Hartshorne.
\newblock Algebraic vector bundles on projective spaces: a problem list.
\newblock {\em Topology}, 18(2):117--128, 1979.

\bibitem{HK}
J{\"u}rgen Herzog and Michael K{\"u}hl.
\newblock On the bettinumbers of finite pure and linear resolutions.
\newblock {\em Communications in Algebra}, 12(13):1627--1646, 1984.

\bibitem{HilbertSyzygy}
David Hilbert.
\newblock Ueber die {T}heorie der algebraischen {F}ormen.
\newblock {\em Math. Ann.}, 36(4):473--534, 1890.

\bibitem{Hulett}
Heather~A. Hulett.
\newblock Maximum {B}etti numbers of homogeneous ideals with a given {H}ilbert
  function.
\newblock {\em Comm. Algebra}, 21(7):2335--2350, 1993.

\bibitem{CraigPaoloJason}
Craig Huneke, Paolo Mantero, Jason McCullough, and Alexandra Seceleanu.
\newblock The projective dimension of codimension two algebras presented by
  quadrics.
\newblock {\em J. Algebra}, 393:170--186, 2013.

\bibitem{HU}
Craig Huneke and Bernd Ulrich.
\newblock The structure of linkage.
\newblock {\em Ann. of Math. (2)}, 126(2):277--334, 1987.

\bibitem{IyWa}
Srikanth~B Iyengar and Mark~E Walker.
\newblock Examples of finite free complexes of small rank and small homology.
\newblock {\em Acta Mathematica}, 221(1):143--158, 2018.

\bibitem{Kunz}
Ernst Kunz.
\newblock Almost complete intersections are not {G}orenstein rings.
\newblock {\em J. Algebra}, 28:111--115, 1974.

\bibitem{kustinchar2}
Andrew~R. Kustin.
\newblock Gorenstein algebras of codimension four and characteristic two.
\newblock {\em Comm. Algebra}, 15(11):2417--2429, 1987.

\bibitem{kustinmillerclass}
Andrew~R Kustin and Matthew Miller.
\newblock Algebra structures on minimal resolutions of gorenstein rings of
  embedding codimension four.
\newblock {\em Mathematische Zeitschrift}, 173(2):171--184, 1980.

\bibitem{KM}
Andrew~R Kustin and Matthew Miller.
\newblock Classification of the tor-algebras of codimension four gorenstein
  local rings.
\newblock {\em Mathematische Zeitschrift}, 190(3):341--355, 1985.

\bibitem{Macaulay}
F.~S. Macaulay.
\newblock {\em The algebraic theory of modular systems}.
\newblock Cambridge Mathematical Library. Cambridge University Press,
  Cambridge, 1994.
\newblock Revised reprint of the 1916 original, With an introduction by Paul
  Roberts.

\bibitem{InfiniteGradedFreeRes}
Jason McCullough and Irena Peeva.
\newblock Infinite graded free resolutions.
\newblock {\em Commutative algebra and noncommutative algebraic geometry},
  1:215--257, 2015.

\bibitem{StillmansConjectureSurvey}
Jason McCullough and Alexandra Seceleanu.
\newblock {\em Bounding Projective Dimension}, pages 551--576.
\newblock Springer New York, New York, NY, 2013.

\bibitem{Moh}
Fatemeh Mohammadi and Johannes Rauh.
\newblock Prime splittings of determinantal ideals.
\newblock {\em Comm. Algebra}, 46(5):2278--2296, 2018.

\bibitem{Munoz}
Vicente Mu{\~n}oz.
\newblock Toral rank conjecture.
\newblock Preprint.

\bibitem{Pardue}
Keith Pardue.
\newblock Deformation classes of graded modules and maximal {B}etti numbers.
\newblock {\em Illinois J. Math.}, 40(4):564--585, 1996.

\bibitem{Santoni}
Larry Santoni.
\newblock Horrocks' question for monomially graded modules.
\newblock {\em Pacific J. Math.}, 141(1):105--124, 1990.

\bibitem{Serre}
Jean-Pierre Serre.
\newblock Sur la dimension homologique des anneaux et des modules
  noeth\'{e}riens.
\newblock In {\em Proceedings of the international symposium on algebraic
  number theory, {T}okyo \& {N}ikko, 1955}, pages 175--189. Science Council of
  Japan, Tokyo, 1956.

\bibitem{W}
Mark~E Walker.
\newblock Total betti numbers of modules of finite projective dimension.
\newblock {\em Annals of Mathematics}, pages 641--646, 2017.

\bibitem{Walkerv1}
Mark~E. Walker.
\newblock Total betti numbers of modules of finite projective dimension,
  ar{X}iv v1.
\newblock {\em arXiv:1702.02560v1}, 2017.

\end{thebibliography}
\end{document}